\documentclass[a4paper,11pt]{amsart}

\usepackage{graphicx, rotating}
\usepackage{subfigure}
\usepackage[font=footnotesize,labelfont=footnotesize]{caption}
\usepackage[utf8]{inputenc}
\usepackage[toc]{appendix}
\usepackage{twoopt}

\usepackage{amssymb}
\usepackage{amsthm}
\usepackage{amsmath,a4wide}
\numberwithin{equation}{section}
\usepackage{hyperref}
\usepackage{amsthm}
\usepackage{mathtools}
\usepackage{amsfonts}
\usepackage[foot]{amsaddr}

\usepackage{marvosym}
\usepackage{verbatim}

\usepackage{enumerate}
\usepackage{pdfsync}
\usepackage{caption}
\usepackage{hyperref}
\usepackage{enumerate}
\mathtoolsset{showonlyrefs}

\theoremstyle{plain}
\newtheorem{theorem}{Theorem}[section]

\theoremstyle{plain}
\newtheorem{corollary}[theorem]{Corollary}

\theoremstyle{plain}
\newtheorem{lemma}[theorem]{Lemma}

\newtheorem*{lemma*}{Lemma}
\theoremstyle{plain}

\newtheorem{proposition}[theorem]{Proposition}
\theoremstyle{definition}

\theoremstyle{remark}

\theoremstyle{remark}

\theoremstyle{definition}

\theoremstyle{plain}
\newtheorem{conjecture}[theorem]{Conjecture}
\theoremstyle{definition}

\newcommand{\R}{\mathbb{R}}
\newcommand{\C}{\mathbb{C}}

\newcommand{\Rd}{\mathbb{R}^d}
\newcommand{\Z}{\mathbb{Z}}
\newcommand{\N}{\mathbb{N}}
\newcommand{\Lt}[1][d]{L^2(\R^{#1})}

\newcommand{\F}{\mathcal{F}}

\renewcommand{\l}{\lambda}
\renewcommand{\L}{\Lambda}

\renewcommand{\H}{\mathbb{H}}

\newcommandtwoopt{\xarrow}[2][0.5cm][0]{\mathrel{\rotatebox[origin=c]{#2}{$\xrightarrow{\rule{#1}{0pt}}$}}}
\usepackage{xcolor}

\begin{document}

\title[Maximal Theta Functions]{
		Maximal Theta Functions\\--\\Universal Optimality of the Hexagonal Lattice for Madelung-Like Lattice Energies
	}
\author[L.~B\'{e}termin]{Laurent B\'{e}termin}
\address{Institut Camille Jordan - Université Claude Bernard Lyon 1, 21 avenue Claude Bernard, 69622 Villeurbanne Cedex, France}
\email{betermin@math.univ-lyon1.fr}

\author[M.~Faulhuber]{Markus Faulhuber}
\address{Faculty of Mathematics, University of Vienna, Oskar-Morgenstern-Platz 1, 1090 Vienna, Austria}
\email{markus.faulhuber@univie.ac.at}

\thanks{Laurent B\'{e}termin was part of the Faculty of Mathematics, University of Vienna, Austria, at the time of writing and was supported by the Vienna Science and Technology Fund (WWTF) MA14-009 as well as the  Austrian Science Fund (FWF) project F65. Markus Faulhuber was with the Department of Mathematics, RWTH Aachen University, Germany at the time of writing and was partially supported by the Vienna Science and Technology Fund (WWTF) VRG12-009 and the Austrian Science Fund (FWF) P33217 and TAI6.}

\subjclass[2010]{}

\begin{abstract}
	We present two families of lattice theta functions accompanying the family of lattice theta functions studied by Montgomery in [H.~Montgomery. Minimal theta functions. \textit{Glasgow Mathematical Journal}, 30(1):75--85, 1988]. The studied theta functions are generalizations of the Jacobi theta-2 and theta-4 functions. Contrary to Montgomery's result, we show that, among lattices, the hexagonal lattice is the unique maximizer of both families of theta functions. As an immediate consequence, we obtain a new universal optimality result for the hexagonal lattice among two-dimensional alternating charged lattices and lattices shifted by the center of their unit cell.
\end{abstract}

\maketitle

\tableofcontents

\section{Introduction and Main Results}\label{sec_Intro}

Theta functions are classical objects appearing in many branches of mathematics and in mathematical physics, e.g., Bose-Einstein Condensates \cite{AftBN,LuoWeiBEC}, vortices in superconductors \cite{Sandier_Serfaty}, chaotic eigenfunctions \cite{NonVor98}, string theory \cite{GannonThesis}, energetically optimal crystals \cite{BetSoftTheta,Bet16,BetKnu_Born_18,BeterminPetrache_DimensionReduction_2017,OptinonCM}, summation of lattice sums \cite{Latticesums}, algebraic geometry \cite{Bost}, transcendence theory \cite{Banaszczyk}, in coding or information theory \cite{ConSlo98, HytVes_Secrecy_2017, StrBea03} or the study of Gabor systems and frames with Gaussian window \cite{Faulhuber_Hexagonal_2018,Faulhuber_SampTA19,FaulhuberSteinerberger_Theta_2017,Jan96} or with hyperbolic secant window \cite{JanssenStrohmer_Secant_2002}.

A strong connection is given between theta functions and the sphere packing problem, a rather simple to understand, but in general hard to crack problem. The variety of applications of the sphere packing problem and connections to other problems is overwhelming. As a starting point, we refer the interested reader to the textbook of Conway and Sloane \cite{ConSlo98}. There are only few dimensions for which the sphere packing problem is completely solved. Of course, the enormous effort of Hales et al.\ to prove the Kepler Conjecture in dimension 3 needs to be mentioned at this point \cite{Hales_2005}, \cite{Hales_2017}, as well as the recent breakthrough of Viazovska \cite{Viazovska8_2017} in dimension 8 and, shortly after, in collaborative work of Cohn, Kumar, Miller, Radchenko and Viazovska in dimension 24 \cite{Viazovska24_2017}.

In \cite{CohKum07}, Cohn and Kumar studied the problem of universal optimality of point configurations on several manifolds, i.e. the optimality of structures among periodic ones with unit density for all Gaussian interactions $r\mapsto e^{-\alpha r^2}$, $\alpha>0$. For $\Rd$, this problem is closely related to, but much harder than the sphere packing problem and was solved in \cite{CohKum07} for dimension $d = 1$. In dimension 8 and 24, the universal optimality of the $\mathsf{E}_8$ and Leech lattice have been proven in \cite{Coh-Via19}. Also, it is remarked in \cite{Coh-Via19} that universal optimality of a point configuration, in particular of a lattice, is a rare feature and thus highly important.

As recalled in \cite[Sect.~9]{CohKum07}, it is known that universal optimality of a lattice cannot hold in dimension 3, 5, 6 and 7, (see \cite{Mar03}) by considering dual lattices of the optimal candidates. Therefore, the only dimensions up to dimension 8 where the problem is open are dimensions 2 and 4.

In dimension 2, the expected candidate to solve the problem of universal optimality (see \cite[Conj.~9.4]{CohKum07}) is the hexagonal lattice, which we denote by $\mathsf{A}_2$ in this work as it results from the $\mathsf{A}_2$ root system. The hexagonal lattice pops up as a known or conjectured optimal solution in several different contexts, and besides the already mentioned literature, it appears in the study of white noise spectrograms \cite{Fla17} (see also \cite{BarFlaCha18}), the study of the heat kernel with varying metric \cite{Baernstein_HeatKernel_1997, Faulhuber_Rama_2020}, or in the study of determinants of Laplace-Beltrami operators \cite{Faulhuber_Determinants_2020, Osgood_Determinants_1988}.

Regarding the expected universal optimality of the hexagonal lattice, we already know, due to the work of Montgomery \cite{Montgomery_Theta_1988}, that the hexagonal lattice is universally optimal among lattices. Montgomery's result is formulated for quadratic forms and he shows that a certain lattice theta function is minimal if and only if the quadratic form comes from the hexagonal lattice. As recalled in \cite{Henn}, this optimality result was previously proved for the Epstein zeta function (i.e., for inverse power-law interactions $r\mapsto r^{-s}$, $s>0$) by Rankin \cite{Ran53}, Ennola \cite{Enn64}, Cassels \cite{Cas59} and Diananda \cite{Dia64} and turns out to be a simple consequence of the universal optimality of $\mathsf{A}_2$ among lattices. Furthermore, the family of theta functions studied by Montgomery can be seen as a generalization of the (lattice) Jacobi $\theta_3$-function.

In this work, we study two families of 2-dimensional lattice theta functions, which are accompanying the family of theta functions studied by Montgomery \cite{Montgomery_Theta_1988}. We call these families of functions the \textit{centered} and \textit{alternating lattice theta functions} which are defined for all $\alpha>0$, respectively, by
%. The centered and the alternating lattice theta function are 

\begin{equation}\label{eq_theta_c}
	\theta^c_{\L} (\alpha) = \sum_{k,l \in \Z} e^{-\tfrac{\pi \alpha}{y} \left( \left( k + \tfrac{1}{2} \right)^2 + 2 x \left(k + \tfrac{1}{2} \right) \left(l + \tfrac{1}{2} \right) + \left(x^2 + y^2\right) \left(l + \tfrac{1}{2} \right)^2 \right)}
\end{equation}
and 
\begin{equation}\label{eq_theta_+-}
	\theta^\pm_\L (\alpha) = \sum_{k,l \in \Z} (-1)^{k+l} e^{-\tfrac{\pi \alpha}{y} \left( k^2 + 2 x k l + \left(x^2 + y^2\right) l^2 \right)},
\end{equation}
for a 2-dimensional lattice $\L$ of volume 1, parametrized by $(x,y) \in \R \times \R_+$. For a detailed explanation on lattice parametrization in dimension 2 see Section \ref{sec:parameters}.

We remark that, in the formulas above, $c$ stands for the center of a \textit{fundamental cell} of the lattice $\L$. This fundamental cell is the parallelogram spanned by the vectors of the Minkowski basis of $\L$. This means that, the 2-dimensional lattice $\theta^c_\L$-function is expressible by \eqref{eq_theta_c} as long as $|x| \leq \tfrac{1}{2}$ and $x^2+y^2 \geq 1$, $y > 0$. Furthermore, the $\pm$ expresses that the sign alternates from one lattice point to another. The presented lattice theta functions $\theta^c$ and $\theta^\pm$ can also be seen as a generalization of the classical Jacobi $\theta_2$ and $\theta_4$ nulls, respectively. This generalization can also be taken to higher dimensions. We will explain the subtleties later on in this work.

By using the Poisson summation formula (and an index transform), it becomes apparent that these functions fulfill the following functional equation;
\begin{equation}\label{eq_fcteqtheta}
	\theta_\L^c(\alpha) = \tfrac{1}{\alpha} \, \theta_\L^\pm \left( \tfrac{1}{\alpha} \right), \quad \forall \alpha>0.
\end{equation}

Our main result on $\theta^c_\L$ and $\theta^\pm_\L$ is analogous, but contrary, to Montgomery's main result.
\begin{theorem}[Main Result]\label{thm_main}
	For any $\alpha > 0$ and any 2-dimensional lattice $\L$ of volume 1, we have
	\begin{equation}
		\theta^c_{\mathsf{A}_2}(\alpha) \geq \theta^c_\L(\alpha)
		\qquad \textnormal{ and } \qquad
		\theta^\pm_{\mathsf{A}_2}(\alpha) \geq \theta^\pm_\L(\alpha)
	\end{equation}
	with equality if and only if $\L$ is a rotated version of the hexagonal lattice $\mathsf{A}_2$, i.e.,
	\begin{equation}
		\L = Q \mathsf{A}_2 \Z^2, \, Q \in SO(2,\R)
		\quad \text{ with } \quad
		\mathsf{A}_2 =
		\sqrt{\tfrac{2}{\sqrt{3}}}
	\begin{pmatrix}
		1 & \tfrac{1}{2}\\
		0 & \tfrac{\sqrt{3}}{2}
	\end{pmatrix}
	\Z^2.
	\end{equation}
\end{theorem}

A remarkable consequence is that Theorem \ref{thm_main} immediately generalizes to a large class of lattice energies. For a lattice $\L = S \Z^2$, with $S = (v_1, v_2)$, where $v_1$, $v_2$ are column vectors yielding a Minkowski basis for $\L$, and for any function $f:\R_+\to \R$, we define the \textit{Madelung-like $f$-energy} of $\L$ by
\begin{equation}\label{eq_Epm}
	E_f^\pm[\L] := \sum_{\substack{(k,l)\in \Z^2\\(k,l)\neq (0,0)}} (-1)^{k+l}f(|kv_1+l v_2|^2).
\end{equation}
The term Madelung-like energy is inspired by its connection to Madelung's constant for the rock-salt structure of Sodium Chloride NaCl (see e.g., \cite{BFK20})
\begin{equation}\label{eq_Madelung_constant}
	M_{\text{Na}} = - M_{\text{Cl}} = \sum_{\substack{(m,n,p) \in \Z^3\\(m,n,p)\neq (0,0,0)}} \frac{(-1)^{m+n+p}}{\sqrt{m^2+n^2+p^2}}.
\end{equation}
This corresponds to the electrostatic energy of an alteration of charges $\pm 1$ on cubic lattice sites originally calculated by Madelung in \cite{Madelung}, which has great similarity to computing energies for other ionic crystals with alternation of charges and different interaction potentials $f$ (see e.g., \cite{BorweinMadelung} for the inverse power law case). We remark that the series in \eqref{eq_Madelung_constant} converges only conditionally and that regularization methods, such as the Ewald summation formula should be used for its computation (see also Section \ref{sec_univopt} and \cite{Tosi}).

We also notice that in \eqref{eq_Epm} as well as in \eqref{eq_Madelung_constant}, unlike in the definition of $\theta_\L^\pm(\alpha)$, the origin is removed from the summation. This choice of defining $E_f^\pm[\L]$ is necessary in order to cover physically relevant interaction potentials with a singularity at $r=0$, like the inverse power laws $f(r)=r^{-s}$, $s>0$. Furthermore, if $f(0)$ exists, its value stays invariant among lattices (and can therefore be ignored).

Similarly, by defining the center of the fundamental cell by $c:=(v_1+v_2)/2$, we define the \textit{lattice center $f$-energy} of $\L$
\begin{equation}
	E_f^c[\L]:=\sum_{(k,l)\in \Z^2} f(|kv_1+l v_2+c|^2).
\end{equation}
In the Gaussian case, this formula is dual to $E_f^\pm$ by the Poisson summations formula. We notice that the origin is not removed from the summation defining $E_f^c[\L]$ since $f(|c|^2)$ is not invariant among lattices and $c\neq 0$, which again allows us to choose an interaction potential $f$ singular at the origin. In both cases, $f: \R_+ \to \R$ is called the potential function, which we assume to be a completely monotone function, i.e., for all $r>0$ and all $k\in \N$, $(-1)^k f^{(k)}(r)\geq 0$.  Assuming sufficiently fast decay at infinity ensures the summability of the above expressions. Following the definition of Cohn and Kumar in \cite{CohKum07}, our main result shows that the hexagonal lattice is \textit{universally optimal among lattices} for this type of energy.

\begin{corollary}[Universal optimality among lattices]\label{cor_cm}
	Let $f$ be a completely monotone function such that $|f(r)|=O(r^{-1-\varepsilon})$ as $r\to +\infty$ for some $\varepsilon>0$. Then, for any 2-dimensional lattice $\L$ of volume 1,
	\begin{equation}
		E_f^\pm[\mathsf{A}_2]\geq E_f^\pm[\L] \qquad \textnormal{ and } \qquad E_f^c[\mathsf{A}_2]\geq E_f^c[\L],
	\end{equation}
	with equality if and only if $\L$ is a rotated version of the hexagonal lattice $\mathsf{A}_2$, i.e.,
	\begin{equation}
		\L = Q \mathsf{A}_2, \, Q \in SO(2,\R)
		\quad \text{ with } \quad
		\mathsf{A}_2 =
		\sqrt{\tfrac{2}{\sqrt{3}}}
	\begin{pmatrix}
		1 & \tfrac{1}{2}\\
		0 & \tfrac{\sqrt{3}}{2}
	\end{pmatrix}
	\Z^2.
	\end{equation}
	Furthermore, for any $2$-dimensional lattice $\L$ of volume 1, we have 
	\begin{equation}\label{eq_Efneq}
	E_f^\pm[\L]<0.
	\end{equation}
\end{corollary}
We believe that \eqref{eq_Efneq} could still hold in higher dimension. Besides being an interesting mathematical property, this kind of result plays a role for instance when looking for ground states among ionic crystals, in particular in the inverse power-law case (see e.g. \cite[Thm. 2.4 and Rmk. 2.5]{BFK20}). Furthermore, we remark that a similar result to Corollary \ref{cor_cm}, involving \textit{local} maximality of the hexagonal lattice among lattices, has been established in \cite{FauSte19}. It might well be true that the result in \cite{FauSte19} holds for completely monotone functions, as the technical assumptions there look as if they can be tailored to suit completely monotone functions.

%\medskip

In analogy to the energy minimization problem considered, e.g., in \cite{Coh-Via19}, where it is also conjectured that the hexagonal lattice is universally optimal, we can conjecture that $\mathsf{A}_2, \mathsf{E}_8$ and the Leech lattice $\L_{24}$ are also universally optimal among periodic configurations with charges $\pm 1$. Following \cite{Coulangeon:2010uq}, we call $\mathcal{A}_N^1$ the space of periodic configurations with unit density such that there is an even number $N\in 2\N$ of points per period, i.e., $\mathcal{A}_N^1$ is the set of all $\mathcal{C}\subset \R^d$ of the form
\begin{equation}\label{eq:periodic}
	\mathcal{C}=\bigsqcup_{i=1}^{N} (t_i+\L), \quad \textnormal{$\L$ is a lattice of covolume $\frac{1}{N}$.}
\end{equation}
Furthermore, defining $T_N:=\{t_1,...,t_{N}\}\subset \R^d$ and given $\mathcal{C}\in \mathcal{A}_N^1$, we call $\Phi_N(\mathcal{C})$ the set of all the periodic and neutral distributions of charges $\{\pm 1\}$ on $\mathcal{C}$, i.e. all the functions
\begin{equation}\label{eq:phi}
\varphi: T_N\to \{-1,1\} \quad \textnormal{such that}\quad \varphi(t_i+u)=\varphi(t_i), \, \forall u\in \L, \quad \textnormal{and}\quad \sum_{i=1}^N \varphi(t_i)=0.
\end{equation}
%We call $\Lambda_N(\mathcal{C})$ the set of all $\varphi$ satisfying \eqref{eq:phi}.

\begin{conjecture}[Universal optimality]
Let $f:(0,\infty)\to \R_+$ be completely monotone and such that $|f(r)|=O(r^{-d/2-\varepsilon})$ as $r\to +\infty$ for some $\varepsilon>0$. For any periodic point configurations $\mathcal{C}\in \mathcal{A}^1_N$ defined by \eqref{eq:periodic}, we define its Madelung $f$-energy by
\begin{equation}
	E_f^\pm[\mathcal{C}]:= \min_{\varphi\in \Phi_N(\mathcal{C})} E_f[\mathcal{C},\varphi],\quad \textnormal{where}\quad  E_f[\mathcal{C},\varphi]:=\frac{1}{N}\sum_{i=1}^N \sum_{q\in \mathcal{C}\backslash \{t_i\}} \varphi(t_i) \varphi(q)  f(|t_i-q|^2).
\end{equation}
Then, for all $N\in 2\N$, $\mathsf{A}_2,\mathsf{E}_8$ and the Leech lattice are the unique maximizer of $E_f^\pm$ on $\mathcal{A}_N^1$ in the respective dimensions $d \in \{2,8,24\}$.
\end{conjecture}

This conjecture is supported by Corollary \ref{cor_cm} as well as the recent work of the first author and Kn\"upfer \cite{BetKnu_Born_18} where the alternate configuration of charges $\{\pm 1\}$ has been shown to minimize $\varphi\mapsto E_f[\L,\varphi]$ in the simple lattice case and for a large set of lattices $\L$ with charges satisfying some simple assumptions. Furthermore, the same conjecture has been stated in \cite{Crystbinary1d} for the one-dimensional case and solved in the same paper for a small class of potentials, including the inverse power laws.

For general countable point sets of a fixed density and the centered $f$-energy, we put a conjecture at the end of this work in Section \ref{sec:parameters}.

\subsection{Preliminaries and Notation}

We will now provide some more background information on different kinds of theta functions and state some preliminary results. We start with the work of Montgomery \cite{Montgomery_Theta_1988}, which is the motivating article for this work, where the following family of lattice theta functions was studied;
\begin{equation}\label{eq_theta_Montgomery}
	\theta_q(\alpha) = \sum_{k,l \in \Z} e^{-\pi \alpha \, q(k,l)}, \quad \alpha \in \R_+,
\end{equation}
where $q(k,l) = a k^2 + b k l + c l^2$ is a positive definite quadratic form with discriminant $\tfrac{b^2}{4} - a c = -1$.\footnote{We have hidden the factor 2 appearing in the exponent in Montgomery's article \cite{Montgomery_Theta_1988} in the discriminant.} We will now explain how Montgomery's family of theta functions is a natural generalization of the lattice Jacobi $\theta_3$-function. This motivates the definition of $\theta_\L^c$ and $\theta_\L^\pm$, which can then be seen as natural generalizations of the Jacobi $\theta_2$ and $\theta_4$-function, respectively.

All functions under consideration are actually restricted Riemann theta functions (see, e.g., \cite[Chap.~II]{TataI}), which are natural generalizations of the classical Jacobi $\Theta_3$- function. A Riemann theta function is a function of the form
\begin{equation}
	\Theta^g(z; \tau) = \sum_{k \in \Z^g} e^{ \pi i k \cdot \tau k} e^{2 \pi i k \cdot z},
\end{equation}
where $z \in \C^g$ and $\tau \in \H^g$, the Siegel upper half-space of genus $g \in \N$, which is an extension of the upper half-plane $\H = \H^1$ (see also \cite[Chap.~II]{TataI});
\begin{equation}
	\H^g = \{ F \in GL(g, \C) \mid F^T = F, \, \Im(F) > 0\}.
\end{equation}
Here, $\Im(F) > 0$ means that the imaginary part of the matrix $F$ is positive definite. Also, we used the notation $z_1 \cdot z_2 = \overline{z_1}^T z_2$ for the inner product of two column vectors $z_1, z_2 \in \C^g$.

If $g=1$, the Riemann theta function is the classical Jacobi $\Theta_3$-function
\begin{equation}
	\Theta_3(z; \tau) = \sum_{k \in \Z} e^{ \pi i \tau k^2} e^{2 \pi i k z} \quad (z,\tau) \in \C \times \H .
\end{equation}
Accompanying, we have the Jacobi $\Theta_2$- and Jacobi $\Theta_4$-function, given by
\begin{equation}
	\Theta_2(z;\tau) = \sum_{k \in \Z} e^{\pi i \tau (k+\frac{1}{2})^2} e^{2 \pi i (k+\frac{1}{2}) z}, \quad (z,\tau) \in \C \times \H
\end{equation}
and
\begin{equation}
	\Theta_4(z;\tau) = \sum_{k \in \Z} (-1)^k e^{\pi i \tau k^2} e^{2 \pi i k z}, \quad (z,\tau) \in \C \times \H ,
\end{equation}
respectively. The Jacobi $\Theta_1$-function is of less significant role for this work and will not be stated or studied here, as its so-called null vanishes identically (because it is an odd function in $z$ and we also refer to the textbook of Whittaker and Watson at this point \cite{WhiWat69}). The so-called nulls (or theta constants) are of special interest for this work. They are obtained by setting $z = 0$;
\begin{align}
	\vartheta_2(\tau) & = \Theta_2(0;\tau) = \sum_{k \in \Z} e^{\pi i \tau (k+\frac{1}{2})^2}\\
	\vartheta_3(\tau) & = \Theta_3(0;\tau) = \sum_{k \in \Z} e^{\pi i \tau k^2}\\
	\vartheta_4(\tau) & = \Theta_4(0;\tau) = \sum_{k \in \Z} (-1)^k e^{\pi i \tau k^2} .
\end{align}
By the Poisson summation formula, we have the functional equations, also known as Jacobi identities;
\begin{equation}
	\vartheta_2(\tau) = \frac{1}{\sqrt{-i \tau}} \, \vartheta_4(-\tfrac{1}{\tau})
	\quad \text{ and } \quad
	\vartheta_3(\tau) = \frac{1}{\sqrt{-i \tau}} \, \vartheta_3(-\tfrac{1}{\tau}).
\end{equation}

In this work, however, we are interested in 2-dimensional lattices. For more details on lattices and their importance throughout different fields (also outside of mathematics) we refer to the textbook of Conway and Sloane \cite{ConSlo98}. We note that a 2-dimensional lattice is a discrete, co-compact subgroup of $\R^2$. It will be no restriction to only consider lattices of (co-)volume 1. Such a lattice can be described by a matrix $S \in SL(2,\R)$;
\begin{equation}
	\L = S \Z^2 = \{k v_1 + l v_2 \mid k,l \in \Z \},
\end{equation}
where $v_1$ and $v_2$ are the columns of $S$. We note that $S$ is not unique, which will be described below and in more detail in Section \ref{sec:parameters}. Nonetheless, by a QR-decomposition, we can always achieve that
\begin{equation}\label{eq_generating_matrix}
	S = y^{-1/2} \, Q
	\begin{pmatrix}
		1 & x\\
		0 & y
	\end{pmatrix}, \qquad Q \in SO(2,\R), \; y > 0.
\end{equation}
The quadratic form associated to the lattice $\L$ is given by the Gram matrix of $S$;
\begin{equation}\label{eq_quadratic_form}
	q_\L(k,l) = (k,l) \, S^T S
	\begin{pmatrix}
		k\\
		l
	\end{pmatrix}
	= \frac{1}{y} \left( k^2 + 2 x k l + (x^2+y^2) l^2 \right).
\end{equation}
In particular, if $S^T S =
	\begin{pmatrix}
		a & \tfrac{b}{2}\\
		\tfrac{b}{2} & c
	\end{pmatrix}
$ has determinant 1, then the quadratic form is $q(k,l) = a k^2 + b kl + c l^2$ and has discriminant $-1$. We note that the quadratic form is independent of the orthogonal matrix $Q \in SO(2,\R)$, as $(QS)^T(QS) = S^T S$. In the sequel, we will only consider lattices with a generator matrix of type \eqref{eq_generating_matrix}. This means that we identify lattices which can be obtained from one another by rotation. In our notation, the lattice theta functions studied by Montgomery in \cite{Montgomery_Theta_1988} are given by
\begin{equation}\label{eq_theta}
	\theta_\L(\alpha) = \sum_{k,l \in \Z} e^{- \pi \tfrac{\alpha}{y} (k^2 + 2x k l +(x^2+y^2)l^2)}.
\end{equation}
If we set
\begin{equation}
	\tau = i \, \alpha \, S^T S, \qquad \alpha > 0,
\end{equation}
then $\tau$ is an element in the Siegel upper half-space of genus 2, i.e., $\tau \in \H^2$, and $\Re(\tau) = 0$. From this point of view, $\theta_q$ defined by \eqref{eq_theta_Montgomery} and $\theta_\L$ defined by \eqref{eq_theta}, which are just different notations for the same function, are restrictions of the Riemann theta function $\Theta^2$;
\begin{equation}
	\theta_q(\alpha) = \theta_\L(\alpha) = \Theta^2 \left(0; \, i \, \alpha \, S^T S \right).
\end{equation}
In analogy to the Jacobi theta functions, we will call $\Theta^2 (0; \tau)$ a Riemann theta null or Riemann theta constant. Also, we note that by the above arguments
\begin{equation}
	\H^2_\Im := \{ \tau \in \H^2 \mid \Re(\tau) = 0 \} \cong \H.
\end{equation}
Thus, we can parametrize a 2-dimensional lattice by $\tau \in \H$, instead of $\tau \in \H^2$, $\Re(\tau) = 0$. The main result in \cite{Montgomery_Theta_1988} has been stated as follows;
\begin{theorem}[Montgomery]
	For any $\alpha > 0$ and any quadratic form $q$ of discriminant $-1$, we have
	\begin{equation}
		\theta_h(\alpha) \leq \theta_q (\alpha),
	\end{equation}
	where $h(k,l) = \frac{2}{\sqrt{3}} \left( k^2 + k l + l^2 \right)$. Equality holds if and only if $q(k,l)$ is integrally equivalent to $h(k,l)$.
\end{theorem}

The quadratic form $h(k,l)$ is derived from the generating matrix of the hexagonal lattice $\mathsf{A}_2$. As announced, we note that the matrix defining a lattice is never unique, because any $2 \times 2$ matrix of determinant 1 with integer entries, which might be called a modular matrix, defines the integer lattice $\Z^2$;
\begin{equation}
	\mathcal{B} \in SL(2, \Z)
	\qquad \Longleftrightarrow \qquad
	\mathcal{B} \Z^2 = \Z^2.
\end{equation}
Furthermore, for any $\widetilde{\tau} \in \H$ there exists a (modular) matrix $\mathcal{B} \in SL(2,\Z)$, such that $\widetilde{\tau} = \mathcal{B} \circ \tau$, with $\tau \in D$. Here, $\mathcal{B}$ acts on an element $\tau \in \H$ by a linear fractional transformation, i.e., $\mathcal{B} \circ \tau = \frac{a \tau + b}{c \tau + d}$, $a,b,c,d \in \Z$, $ad-bc = 1$ and
\begin{equation}\label{eq_D}
	D = \{ \tau \in \H \mid |\tau| \geq 1, \ |\Re(\tau)| \leq \tfrac{1}{2} \}
\end{equation}
is the fundamental domain of $\H$. For details we refer to the textbook of Serre \cite{Serre} and Section \ref{sec:parameters}. Also, we remark that for our purposes (due to symmetry reasons), it suffices to work in the right half of the fundamental domain $D$, i.e.,
\begin{equation}\label{eq_D+}
	D_+ = \{ \tau \in \H \mid |\tau| \geq 1, \ 0 \leq \Re(\tau) \leq \tfrac{1}{2} \}.
\end{equation}

After these preliminaries, we see that (after a suitable choice of basis and rotation) any lattice $\L$ can be written as
\begin{equation}
	\L = S \Z^2 = S (\mathcal{B} \Z^2) = (S \mathcal{B}) \Z^2,
\end{equation}
with $S = \tfrac{1}{\sqrt{y}}
\begin{pmatrix}
	1 & x\\
	0 & y
\end{pmatrix}
$, $0 \leq x \leq \tfrac{1}{2}$ and $x^2+y^2 \geq 1$, $y \in \R_+$. Under these conditions, we are working with a Minkowski basis of the lattice, which is important for our work.

In general, the quadratic forms derived from the Gram matrices of $S$ and $S \mathcal{B}$ are called integrally equivalent (see also \cite{ConSlo98}). They result from the same lattice, but different bases were used. Hence, Montgomery's main result in \cite{Montgomery_Theta_1988} states that the hexagonal lattice uniquely minimizes $\theta_\L(\alpha)$ (among lattices of unit volume) for any fixed $\alpha > 0$.

\medskip
After this crash course on lattice structures and how to parametrize them, we draw connections to 1-dimensional lattice theta functions. First, we note that Montgomery's 2-dimensional lattice theta function defined in \eqref{eq_theta} is a very natural extension of the 1-dimensional lattice-$\theta_3$ function\footnote{We note that the only 1-dimensional lattice structure we have is (isomorphic to) $\Z$.};
\begin{equation}
	\theta_3(t) = \sum_{k \in \Z} e^{- \pi t k^2}, \qquad t \in \R_+.
\end{equation}
We also note the analogy in the respective functional equations;
\begin{equation}
	\theta_3(t) = \tfrac{1}{\sqrt{t}} \, \theta_3 \left( \tfrac{1}{t} \right)
	\qquad \textnormal{ and } \qquad
	\theta_\L(\alpha) = \tfrac{1}{\alpha} \, \theta_\L \left( \tfrac{1}{\alpha} \right).
\end{equation}
The equation for $\theta_3$ is a special case of the Jacobi identity for Jacobi's $\vartheta_3$-function. The identity for $\theta_\L$ would usually involve the dual lattice $\L^*$ which is given by
\begin{equation}
	\L^* = S^{-T} \Z^2.
\end{equation}
However, by a re-labeling of the indices we derive the result.  Equivalently, we may use the symplectic version of the Poisson summation formula. In this case, the standard symplectic form $\sigma(. \, , .)$
\begin{equation}
	\sigma((x_1,y_1),(x_2,y_2)) = x_1 y_2 - x_2 y_1, \quad (x_1,y_1), (x_2, y_2) \in \R^2,
\end{equation}
which is skew-symmetric, replaces the usual standard Euclidean inner product, denoted by $\cdot$, in the complex exponential. The symplectic Poisson summation formula involves the symplectic Fourier transform, which only exists for functions of $2d$ variables, just as the symplectic form\footnote{The general standard symplectic form on $\R^{2d}$ is given by $\sigma(z_1, z_2) = x_1 \cdot y_2 - x_2 \cdot y_1$, $z_1=(x_1,y_1)$, $z_2=(x_2,y_2)$, $x_1,x_2,y_1,y_2 \in \Rd$.}. For dimension 2 it is given by
\begin{equation}\label{eq_sympfourier}
	\mathcal{F}_\sigma F(\xi,\eta) = \int_{\R^2} F(\xi', \eta') e^{-2 \pi i \sigma\left((\xi,\eta),(\xi',\eta')\right)} \, d\xi' d\eta'.
\end{equation}
This formula holds for Schwartz functions and extends to $\Lt[2]$, just as the usual Fourier transform. The symplectic Poisson summation formula reads \cite{Faulhuber_Note_2018}
\begin{equation}
	\sum_{\l \in \L} F(\l + z) = \sum_{\l^\circ \in \L^\circ} \F_\sigma F(\l^\circ) e^{2 \pi i \sigma(\l^\circ, z)}.
\end{equation}
The lattice $\L^\circ = J \L^*$ is the so-called adjoint lattice, commonly used in time-frequency analysis (see, e.g., \cite{Gro01}). The advantage is that in dimension 2 any lattice is symplectic, i.e., the generating matrix fulfills,
\begin{equation}
	S J S^T = J, \quad
	J = \begin{pmatrix}
		0 & 1\\
		-1 & 0
	\end{pmatrix}.
\end{equation}
For lattices of the form $\L = S \Z^2$, $S$ symplectic, we have $\L^\circ = J S^{-T} (J^{-1} \Z^2) = \L$. Furthermore, the Gaussian functions $e^{-\pi q_\L(\xi,\eta)}$, where $(\xi,\eta) \in \R^2$ and $q_\L$ is the quadratic form defined by \eqref{eq_quadratic_form}, are eigenfunctions of the symplectic Fourier transform with eigenvalue 1. For this fact and an explicit usage of the symplectic Poisson summation formula we refer to \cite{Faulhuber_Note_2018}. For details on symplectic methods in harmonic analysis as well as in mathematical physics we refer to \cite{Fol89} and \cite{Gos11}.

Furthermore, we note that the lattice $\theta_3$-function is a restriction of the Jacobi $\Theta_3$-null, i.e., $\vartheta_3(\tau)$, if $\tau = i t$, $ t \in \R_+$. We can also say that we pick $\tau$ from
\begin{equation}
	\H_\Im := \{ \tau \in \H \mid \Re(\tau) = 0 \}.
\end{equation}

In analogy to the extension of the Jacobi $\Theta_3$ function, which yields the Riemann theta function $\Theta^g$, we could as well define the Riemann $\Theta_2^g$- and $\Theta_4^g$-function in the following way;
\begin{align}
	\Theta_2^g (z; \tau) & = \sum_{k \in \Z^g} e^{\pi i \left(k+\mathbf{\tfrac{1}{2}}\right) \cdot \tau \left(k+\mathbf{\tfrac{1}{2}}\right)} e^{2 \pi i k \cdot z}, & & \mathbf{\tfrac{1}{2}} = (\tfrac{1}{2}, \ldots , \tfrac{1}{2}) \in \R^g,\\
	\Theta_4^g (z; \tau) & = \sum_{k \in \Z^g} (-1)^{k_1 + \ldots + k_g} e^{\pi i k \cdot \tau k} e^{2 \pi i k \cdot z}, & & (z,\tau) \in \C^g \times \H^g.
\end{align}
Hence, $\theta^c_\L$ and $\theta^\pm_\L$ as defined in \eqref{eq_theta_c} and \eqref{eq_theta_+-} are possible extensions of the 1-dimensional lattice $\theta_2$ and $\theta_4$-functions,
\begin{equation}
	\theta_2(t) = \sum_{k \in \Z} e^{-\pi t \left(k +  \tfrac{1}{2} \right)^2},
	\quad \text{ and} \quad
	\theta_4(t) = \sum_{k \in \Z} (-1)^k e^{- \pi t k^2},
\end{equation}
to 2-dimensional lattices. We note that, as a result of the Poisson summation formula (or the Jacobi identity), we have the functional equation
\begin{equation}\label{eq_fcteqtheta24}
	\theta_2(t) = \tfrac{1}{\sqrt{t}} \, \theta_4 \left( \tfrac{1}{\sqrt{t}} \right).
\end{equation}
As a result of the symplectic Poisson summation formula, we have the Jacobi-like identity
\begin{equation}
	\theta^c_\L(\alpha) = \tfrac{1}{\alpha} \, \theta^\pm_\L \left(\tfrac{1}{\alpha} \right).
\end{equation}

Before starting to prove the results, we introduce some more notation and state the product representation for the Jacobi $\Theta_3$-function $(z,\tau) \in \C \times \H$, as we will use it later on.
\begin{align}
		\Theta_3(z; \tau) & = \sum_{k \in \Z} e^{\pi i k^2 \tau} e^{2 \pi i k z}\\
		 & = \prod_{k \geq 1} \left( 1 - e^{2k \pi i \tau} \right) \left( 1 + e^{(2k-1) \pi i \tau} e^{2 \pi i z} \right) \left( 1 + e^{(2k-1) \pi i \tau} e^{-2 \pi i z}\right).
\end{align}
Details on the equality of the series and infinite product representation are given, e.g., in the textbook of Stein and Shakarchi \cite{SteSha_Complex_03}. For purely imaginary $\tau$ and real $z$ the above function is real-valued and (up to re-scaling) the fundamental solution to the heat equation on the circle line. Consider this restricted, real-valued 1-dimensional lattice theta function
\begin{equation}\label{eq_hk_Fourier}
	\widehat{\vartheta}(\beta; t) = \sum_{k \in \Z} e^{- \pi t k^2} e^{ 2 \pi i k \beta}, \quad (\beta,t) \in \R \times \R_+ \, .
\end{equation}
The usage of the notion $\widehat{\vartheta}$ refers to the fact that we are dealing with a Fourier series. We define the related 1-dimensional lattice theta function
\begin{equation}\label{eq_hk_shift}
	\vartheta(\beta; t) = \sum_{k \in \Z} e^{- \pi t (k+\beta)^2}.
\end{equation}
By the Poisson summation formula, we have the identity
\begin{equation}
	\widehat{\vartheta}(\beta; t) = \tfrac{1}{\sqrt{t}} \, \vartheta \left(\beta; \tfrac{1}{t} \right) .
\end{equation}
Furthermore, note that $\F (e^{-\pi (x+\beta)^2})(\omega) = e^{-\pi \omega^2} e^{2 \pi i \omega \beta}$, where $\F$ denotes the Fourier transform
\begin{equation}
	\F f(\omega) = \int_{\R} f(x) e^{-2 \pi i x \omega} \, dx,
\end{equation}
which is another reason why we chose the $\widehat{\vartheta}$ -notation.

It is not hard to see that
\begin{equation}
	\theta_2(t) = \vartheta \left( \tfrac{1}{2},t \right)
	\qquad \textnormal{ and } \qquad
	\theta_4(t) = \widehat{\vartheta} \left( \tfrac{1}{2},t \right).
\end{equation}
Also,
\begin{equation}
	\theta_3(t) = \widehat{\vartheta}(0;t) = \vartheta(0;t).
\end{equation}

In analogy to $\vartheta$ and $\widehat{\vartheta}$, we define the 2-dimensional lattice theta functions\footnote{We could have stressed the fact that $\vartheta$ and $\widehat{\vartheta}$ are 1-dimensional lattice theta functions by writing the lattice $\Z$ in the index, i.e., writing $\vartheta_\Z$ and $\widehat{\vartheta}_\Z$.}
\begin{equation}
	\Theta_\L(\xi,\eta;\alpha) = \sum_{k,l \in \Z} e^{- \tfrac{\pi \alpha}{y} \left( (k+\xi)^2 + 2 x (k+\xi)(l+\eta) + (x^2+y^2) (l+\eta)^2 \right)}
\end{equation}
and
\begin{equation}\label{eq_hat_Theta_L}
	\widehat{\Theta}_\L(\xi,\eta;\alpha) = \sum_{k,l \in \Z} e^{-\tfrac{\pi \alpha}{y} (k^2 + 2 x k l + (x^2+y^2) l^2)} e^{2 \pi i (k \eta - l \xi)}.
\end{equation}

Note that we used the symplectic form $\sigma$ explicitly in the definition of \eqref{eq_hat_Theta_L}. We have the following functional equation, due to the symplectic Poisson summation formula;
\begin{equation}\label{eq_functional}
	\Theta_\L(\xi, \eta; \alpha) = \tfrac{1}{a} \, \widehat{\Theta}_\L(\xi,\eta; \tfrac{1}{\alpha}).
\end{equation}
The function $\theta_\L$, given by \eqref{eq_theta} and considered by Montgomery, is obtained as a special case of the above functions;
\begin{equation}
	\theta_\L(\alpha) = \Theta_\L(0,0;\alpha) = \widehat{\Theta}_\L(0,0;\alpha).
\end{equation}
Furthermore, we also obtain the $\theta_\L^c$ and the $\theta_\L^\pm$ function as special cases;
\begin{equation}
	\theta_\L^c(\alpha) = \Theta_\L \left( \tfrac{1}{2}, \tfrac{1}{2}; \alpha \right)
	\qquad \textnormal{ and } \qquad
	\theta_\L^\pm = \widehat{\Theta}_\L \left( \tfrac{1}{2}, \tfrac{1}{2}; \alpha \right).
\end{equation}

\section{Auxiliary Technical Results}

\subsection{Montgomery's Results}\label{sec_Mont}
In this section we present some of the results of Montgomery's work \cite{Montgomery_Theta_1988}, which are integral for this article.
%\begin{lemma*}[First Main Lemma, Montgomery \cite{Montgomery_Theta_1988}]
%	If $\alpha > 0$, $0 < x < \tfrac{1}{2}$ and $y \geq \tfrac{1}{2}$, then
%	\begin{equation}
%		\frac{\partial}{\partial x} \theta_L(\alpha) < 0.
%	\end{equation}
%\end{lemma*}

We note that, for any parameter $t > 0$, the above 1-dimensional lattice theta functions $\widehat{\vartheta}$ and $\vartheta$, defined by \eqref{eq_hk_Fourier} and \eqref{eq_hk_shift} respectively, are periodic in $\beta$ with period 1 and also are even functions of $\beta$. Furthermore, from the product representation, we conclude that the maximum is achieved for $\beta \in \Z$ and the minimum for $\beta \in \Z + \tfrac{1}{2}$;
\begin{align}
	\widehat{\vartheta}(\beta;t) & = \prod_{k \geq 1} \left( 1 - e^{- 2 k \pi t} \right) \left( 1 + e^{-(2k -1) \pi t} e^{2 \pi i \beta} \right)\left( 1 + e^{-(2k -1) \pi t} e^{-2 \pi i \beta} \right)\\
	& =	\prod_{k \geq 1} \left( 1 - e^{- 2 k \pi t} \right) \left( 1 + 2 e^{- (2k-1) \pi t} \cos(2 \pi \beta) + e^{- (4k-2) \pi t} \right).
\end{align}
The same behavior is shown by the following auxiliary function given in \cite{Montgomery_Theta_1988};
\begin{equation}
	Q(\beta; t) = - \frac{\tfrac{\partial}{\partial \beta} \, \widehat{\vartheta}(\beta; t)}{\sin(2 \pi \beta)}.
\end{equation}

\begin{lemma}[Auxiliary Lemma 1, Montgomery \cite{Montgomery_Theta_1988}]\label{lem_aux1_Mont}
	Let $Q(\beta; t)$ be as above and $t > 0$ fixed. Then $Q(\beta; t)$ is an even function of $\beta$ with period 1 and all values are positive. Furthermore, $Q(\beta; t)$ is a strictly decreasing function of $\beta$ on the interval $\left( 0,\tfrac{1}{2} \right)$.
\end{lemma}
The proof is given in \cite{Montgomery_Theta_1988} and makes use of the product representation of $\vartheta(\beta; t)$. By using l'H\^o{}pital's rule, we see that $Q(\beta;t)$ is a well-defined function on $\R \times \R_+$. The next auxiliary lemma bounds the function $Q(\beta; t)$ from above and from below.
\begin{lemma}[Auxiliary Lemma 2, Montgomery \cite{Montgomery_Theta_1988}]\label{lem_aux2_Mont}
	Let $Q(\beta; t)$ be as above. We define the functions
	\begin{equation}
		A(t) =
		\begin{cases}
			t^{-3/2} e^{-\tfrac{\pi}{4 t}}, & \, 0 \leq t < 1\\
			\left(1 - \tfrac{1}{3000} \right) 4 \pi \, e^{-\pi t}, & \, 1 \leq t
		\end{cases} \; ,
	\end{equation}
	and
	\begin{equation}
		B(t) =
		\begin{cases}
			t^{-3/2}, & \, 0 \leq t < 1\\
			\left(1 + \tfrac{1}{3000} \right) 4 \pi \, e^{-\pi t}, & \, 1 \leq t
		\end{cases} \; .
	\end{equation}
	Then, for any $\beta \in \R$ and any $t\geq 0$, we have
	\begin{equation}
		A(t) \leq Q(\beta; t) \leq B(t).
	\end{equation}
\end{lemma}
The proof is given in \cite{Montgomery_Theta_1988} and makes use of the rapid convergence of the series and the fact that the first term contributes the main amount to the series. Furthermore, the fact that $Q(0; t)$ yields the maximal value and $Q \left( \tfrac{1}{2}; t \right)$ yields the minimal value allows for the estimates $B(t)$ and $A(t)$ respectively. We will use similar arguments in the sequel.

\subsection{Results for the proof of the first main lemma}
We will now state and prove additional auxiliary results, which we will need in order to prove our first main lemma in Section \ref{sec_Main1}.

Consider the classical Jacobi theta-2 function of the two complex variable $\tau \in \H$ and $z \in \C$;
\begin{align}\label{eq_theta2}
	\theta_2(z;\tau) & = \sum_{k \in \Z} e^{\pi i \tau \left(k +\tfrac{1}{2} \right)^2} e^{2 \pi i \left( k + \tfrac{1}{2} \right) z}\\
	& = 2 \, e^{\pi i \tfrac{\tau}{4}} \, \cos(\pi z) \prod_{k \geq 1} \left(1 - e^{2 \pi i \tau k}\right) \left( 1 + 2 \cos(2 \pi z) e^{2 \pi i \tau k} + e^{4 \pi i \tau k}\right).
\end{align}
Restricting the variables to $\tau = i t$, $t \in \R_+$ and $z = \beta \in \R$, we denote the resulting function by
\begin{equation}
	\widehat{\vartheta}_2(\beta; t) = \theta_2(\beta; i t) = \sum_{k \in \Z} e^{-\pi t \left(k +\tfrac{1}{2} \right)^2} e^{2 \pi i \left( k + \tfrac{1}{2} \right) \beta} = \sum_{k \in \Z} e^{-\pi t \left(k +\tfrac{1}{2} \right)^2} \cos(2\pi \left(k + \tfrac{1}{2} \right) \beta)
\end{equation}
Obviously, this function is periodic in $\beta$ with period 2.

By using the Poisson summation formula, we can define the function
\begin{equation}\label{eq_theta2+-}
	\vartheta_2(\beta;t) = \sum_{k \in \Z} (-1)^k e^{- \pi t (k+\beta)^2},
\end{equation}
analogous to $\vartheta$ from Section \ref{sec_Mont}. These functions fulfill the functional equation
\begin{equation}\label{eq_theta2_functional}
	\vartheta_2(\beta; t) = \tfrac{1}{\sqrt{t}} \, \widehat{\vartheta}_2 \left(\beta; \tfrac{1}{t} \right),
\end{equation}
The notation $\widehat{\vartheta}_2$ and $\widehat{\vartheta}$ we use here, stresses the fact that we deal with two sides of the same medal. Turning the medal amounts to using the Poisson summation formula. Depending on the specific situation, one side will be more useful to work with than the other. The functions $\vartheta_2$ and $\widehat{\vartheta}_2$ are derived from a periodization of Gaussian functions, i.e., the heat kernel on $\R$, with alternating sign and shifted heat kernel on the torus, respectively. The connection between the heat kernel on $\R$, on the torus and theta functions is also explained in \cite{Jorgenson_Heat_2001} (even more generally for manifolds and quotient spaces).

We note that the function $\widehat{\vartheta}$ considered in Section \ref{sec_Mont} and by Montgomery \cite{Montgomery_Theta_1988} is actually the classical Jacobi theta-3 function with restricted arguments.

We define the following auxiliary function;
\begin{equation}\label{eq:defQ2}
	Q_2 (\beta; t) := - \frac{\partial_\beta \widehat{\vartheta}_2(\beta; t)}{\sin(\pi \beta)}, \qquad \beta \in \R, \, t \in \R_+.
\end{equation}
\begin{lemma}\label{lem_aux1}
	The function $Q_2(\beta; t)$ is an even function of $\beta$ with period 1 and all values are positive. Furthermore, $Q_2(\beta;t)$ is a strictly decreasing function of $\beta$ on the interval $\left(0, \tfrac{1}{2} \right)$.
\end{lemma}
\begin{proof}
	We use the product identity from \eqref{eq_theta2} and differentiate with respect to $\beta$;
	\begin{align}
		\partial_\beta \widehat{\vartheta}_2(\beta;t) & = 2 \, e^{-\pi \tfrac{t}{4}} \, \left( -\pi \sin(\pi \beta) \prod_{k \geq 1} \left(1 - e^{-2 \pi t k}\right) \left( 1 + 2 \cos(2 \pi \beta) e^{-2 \pi t k} + e^{-4 \pi t k}\right) \right.\\
		& -4 \pi \cos(\pi \beta) \sin(2 \pi \beta)  \sum_{l \geq 1} \left(1 - e^{-2 \pi t l}\right) e^{-2 \pi t l} \\
		& \qquad \left. \times \prod_{\substack{k \geq 1\\k \neq l}} \left(1 - e^{-2 \pi t k}\right) \left( 1 + 2 \cos(2 \pi \beta) e^{-2 \pi t k} + e^{-4 \pi t k}\right) \right)
	\end{align}

By using the fact that $\frac{\sin(2\pi \beta)}{\sin(\pi \beta)} = 2 \cos(\pi \beta)$, we see that
\begin{align}
	Q_2(\beta;t) & = 2 \,\pi \, e^{-\pi \tfrac{t}{4}} \, \left(\prod_{k \geq 1} \left(1 - e^{-2 \pi t k}\right) \left( 1 + 2 \cos(2 \pi \beta) e^{-2 \pi t k} + e^{-4 \pi t k}\right) \right.\\
		& + 8 \cos(\pi \beta)^2 \sum_{l \geq 1} \left(1 - e^{-2 \pi tl}\right) e^{-2 \pi t l} \\
		& \qquad \left. \times \prod_{\substack{k \geq 1\\k \neq l}} \left(1 - e^{-2 \pi t k}\right) \left( 1 + 2 \cos(2 \pi \beta) e^{-2 \pi t k} + e^{-4 \pi t k}\right) \right). 
\end{align}
Note, that the division by $\sin(\pi \beta)$ for $\beta \in \Z$ is justified by l'H\^o{}pital's rule.

The above expression is even with respect to $\beta$ and periodic in $\beta$ with period 1. For $t > 0$ and $\beta \in \left[0, \tfrac{1}{2} \right]$ all terms in brackets are positive. Also, as both $\cos(\pi \beta)^2$ and $\cos(2 \pi \beta)$ are strictly decreasing on $\left(0, \tfrac{1}{2} \right)$, the same is true for $Q_2(\beta;t)$ as a function of $\beta$.
\end{proof}

\begin{lemma}\label{lem_aux2}
	Let $Q_2(\beta;t)$ be as in \eqref{eq:defQ2}. We define the functions
	\begin{equation}
		A_2(t) =
		\begin{cases}
			2 \pi \left(1- \tfrac{1}{175} \right) \, t^{-3/2} \, e^{-\tfrac{\pi}{4 t}}, & \, 0 \leq t < 1\\
			2 \pi \left(1- \tfrac{1}{175} \right) e^{-\tfrac{\pi t}{4}}, & \, 1 \leq t
		\end{cases} \; ,
	\end{equation}
	and
	\begin{equation}
		B_2(t) =
		\begin{cases}
			\pi \, t^{-3/2}, & \, 0 \leq t < 1\\
			2 \pi \left(1 + \tfrac{1}{55} \right) e^{-\tfrac{\pi t}{4}}, & \, 1 \leq t
		\end{cases} \; .
	\end{equation}
	Then, for any $\beta \in \R$, we have
	\begin{equation}
		A_2(t) \leq Q_2(\beta;t) \leq B_2(t).
	\end{equation}
\end{lemma}
\begin{proof}
	From Lemma \ref{lem_aux1} we know that
	\begin{equation}
		Q_2\left( \tfrac{1}{2}; t \right) \leq Q_2(\beta; t) \leq Q_2(0;t).
	\end{equation}
	Hence, it suffices to show the inequalities
	\begin{equation}
		A_2(t) \leq Q_2\left( \tfrac{1}{2}; t \right)
		\qquad \textnormal{ and } \qquad
		Q_2(0;t) \leq B_2(t).
	\end{equation}
	By definition of $Q_2$, we see that 
	\begin{equation}
		Q_2\left( \tfrac{1}{2}; t \right) = - \partial_\beta \widehat{\vartheta}_2(\beta;t) \Bigg|_{\beta = \tfrac{1}{2}}.
	\end{equation}
	First, we assume that $ t \geq 1$. We differentiate $\widehat{\vartheta}_2$ with respect to $\beta$ to get
	\begin{align}
		Q_2\left( \tfrac{1}{2}; t \right) & = 2 \pi \sum_{k \in \Z} \sin\left(\pi \left(k + \tfrac{1}{2}\right) \right) \left(k+ \tfrac{1}{2}\right) \, e^{- \pi t \left(k + \tfrac{1}{2} \right)^2} \\
		& = 2 \pi \sum_{k \in \Z} (-1)^k \left(k+ \tfrac{1}{2}\right) \, e^{- \pi t \left(k + \tfrac{1}{2} \right)^2}\\
		& \geq 2 \pi \, e^{-\tfrac{\pi t}{4}} \left( 1 - \sum_{k \geq 1} \left(2k + 1\right) \, e^{- \pi t \left(k^2 + k \right)} \right).
	\end{align}
	For the last inequality we paired the terms of $k$ with terms of $-k-1$. Now, we note that the last series is decreasing in $t$ and, hence, as $t \geq 1$ get the following estimate; 
	\begin{equation}
		\sum_{k \geq 1} \left(2k + 1\right) \, e^{- \pi t (k^2+k)} \leq \sum_{k \geq 1} \left(2k + 1\right) \, e^{- \pi (k^2+k)} < 0.00560237 < \frac{1}{175}.
	\end{equation}
	Therefore, for $t \geq 1$ we have
	\begin{equation}
		Q_2(\beta;t) \geq 2 \pi \left(1- \tfrac{1}{175} \right) e^{-\tfrac{\pi t}{4}}=A_2(t).
	\end{equation}
	For the case $t < 1$, we use the functional equation \eqref{eq_theta2_functional} which allows us to switch to estimating $\vartheta_2$, given by \eqref{eq_theta2+-}, for $t > 1$.
	\begin{align}
		Q_2\left( \tfrac{1}{2}; \tfrac{1}{t} \right) & = - \partial_\beta \, \sqrt{t} \, \vartheta_2 \left( \beta; t \right) \Bigg|_{\beta = \tfrac{1}{2}}
		= 2 \pi \, t^{3/2} \sum_{k \in \Z} (-1)^k \left(k + \tfrac{1}{2} \right) e^{- \pi t \left(k + \tfrac{1}{2}\right)^2}\\
		& = 2 \pi \, t^{3/2} \, e^{-\tfrac{\pi t}{4}} \left(1 - \sum_{k \geq 1} \left(2k + 1\right) \, e^{- \pi t \left(k^2 + k \right)} \right) \geq 2 \pi \, t^{3/2} \, e^{-\tfrac{\pi t}{4}} \left(1 - \tfrac{1}{175} \right)=A_2(t),
	\end{align}
	where we have used the same estimates as in the $t\geq 1$ case above.
	
	Next, we will prove that $B_2$ bounds $Q_2$ from above. First, we note that the value of $Q_2(0;t)$ cannot be evaluated directly by using the series representation, as both, numerator and denominator vanish. Hence, we need to apply l'H\^o{}pital's rule and see that, for $t \geq 1$, we have
	\begin{align}
		Q_2(0; t) & = - \frac{1}{\pi} \, \partial_\beta^2 \, \widehat{\vartheta}_2(\beta;t) \Bigg|_{\beta = 0} = 4 \pi \sum_{k \in \Z} \left(k + \tfrac{1}{2} \right)^2 e^{- \pi t \left(k + \tfrac{1}{2} \right)^2} \\
		& \leq 2 \pi \, e^{-\tfrac{\pi t}{4}} \left( 1 + \sum_{k \geq 1} (2k+1)^2 e^{-\pi \left( k^2 + k \right)} \right) < 2 \pi \, e^{-\tfrac{\pi t}{4}} \left( 1 + \tfrac{1}{55}\right)=B_2(t).
	\end{align}
	For $t < 1$, we use the functional equation \eqref{eq_theta2_functional} again to use $\vartheta_2$ for $t > 1$.
	\begin{align}
		Q_2 \left(0;\tfrac{1}{t}\right) & = - \frac{1}{\pi} \, \partial_\beta^2 \, \sqrt{t} \, \vartheta_2(\beta;t) \Bigg|_{\beta = 0} = 2 \, t^{3/2} \, \sum_{k \in \Z} (-1)^k \left( 1 - 2 \pi \, t \, k^2 \right) e^{-\pi t k^2}\\
		& \leq 2 \, t^{3/2} \, \left(1 + 2 \sum_{k \geq 0} 2 \pi \, t \, (2k + 1)^2 \,  e^{- \pi (2k + 1)^2 t} \right)
	\end{align}
	To obtain the last inequality, we note that the term $k = 0$ contributes an amount of 1 and that the contributions for $\pm k$ even are negative, hence, we only needed to consider terms with $\pm k$ odd. Next, we observe that
	\begin{equation}
		\partial_t( t \, e^{- \alpha t}) < 0, \qquad \textnormal{ for } \frac{1}{\alpha} < t.
	\end{equation}
	Hence, for $t > 1$ and $k \geq 0$ the expression $t \, e^{- \pi (2k + 1)^2 t}$ is decreasing with respect to $t$. Therefore, we get the estimate
	\begin{equation}
		Q_2(0;t) \leq 2 \, t^{3/2} \left(1 + 2 \sum_{k \geq 0} 2 \pi (2k + 1)^2 \, e^{- \pi (2k + 1)^2} \right) < 2 \, t^{3/2} \underbrace{(1 + 0.55)}_{< \tfrac{\pi}{2}} < \pi \, t^{3/2}=B_2(t).
	\end{equation}
\end{proof}

\subsection{Results for the proof of the second main lemma}
We now state further auxiliary results which we will need for the proof of our second main lemma in Section \ref{sec_Main2}. We start with a result from \cite{FaulhuberSteinerberger_Theta_2017}.
\begin{proposition}[Faulhuber, Steinerberger \cite{FaulhuberSteinerberger_Theta_2017}]\label{pro_fauste}
	Let $t > 0$ be fixed and let $y \in \R_+$. Then, the following inequalities hold;
	\begin{equation}
		\theta_2(t y) \, \theta_2 \left( \tfrac{t}{y} \right) \leq \theta_2(t)^2
		\qquad \textnormal{ and } \qquad
		\theta_4(t y) \, \theta_4 \left( \tfrac{t}{y} \right) \leq \theta_4(t)^2.
	\end{equation}
	In both cases, equality holds if and only if $y = 1$. Furthermore, for $y > 1$,
	\begin{equation}
		\frac{\partial}{\partial y} \left[ \theta_2(t y) \, \theta_2 \left( \tfrac{t}{y} \right) \right] < 0
		\qquad \textnormal{ and } \qquad
		\frac{\partial}{\partial y} \left[ \theta_4(t y) \, \theta_4 \left( \tfrac{t}{y} \right) \right] < 0.
	\end{equation}
	For $ 0 < y < 1$, the expressions are monotonically increasing.
\end{proposition}
For completeness, we remark that in \cite{FaulhuberSteinerberger_Theta_2017} it was also proven that the inequalities are reversed for the expression $\theta_3(t y) \, \theta_3 \left( \tfrac{t}{y} \right)$. This was also established by Montgomery \cite{Montgomery_Theta_1988}.

Another result we will need is the following.
\begin{lemma}\label{lem_1/4}
	For $t > 0$ and $k \in \Z$, we have
	\begin{equation}
		\vartheta \left( \tfrac{1}{4}; t \right) = \vartheta \left( \tfrac{1}{4} + \tfrac{k}{2}; t \right).
	\end{equation}
\end{lemma}
\begin{proof}
	This is actually a corollary of Montgomery's lemma, which we stated as Lemma \ref{lem_aux1_Mont} in this work.
\end{proof}

\begin{lemma}\label{lem_1/4_theta2}
	For $t > 0$, we have
	\begin{equation}
		\vartheta \left( \tfrac{1}{4}; t \right) = \frac{1}{2} \, \theta_2 \left( \tfrac{t}{4} \right)
	\end{equation}
\end{lemma}
\begin{proof}
	The proof makes use of the Poisson summation formula and pairing indices $k$ and $-k$. We compute
	\begin{align}
		\vartheta \left( \tfrac{1}{4}; t \right) & = \sum_{k \in \Z} e^{-\pi t \left(k + \tfrac{1}{4} \right)^2} = \frac{1}{\sqrt{t}} \, \sum_{k \in \Z} e^{-\tfrac{\pi}{t} k^2} e^{2 \pi i \tfrac{k}{4}} = \frac{1}{\sqrt{t}} \, \sum_{k \in \Z} e^{-\tfrac{\pi}{t} k^2} i^k\\
		& = \frac{1}{\sqrt{t}} \, \sum_{k \in \Z} e^{- \tfrac{\pi}{t} (2k)^2} (-1)^k = \frac{1}{2} \, \sum_{k \in \Z} e^{- \pi \tfrac{t}{4} \left(k + \tfrac{1}{2} \right)^2}.
	\end{align}
\end{proof}

\section{The first main lemma}\label{sec_Main1}
By combining the two auxiliary results, Montgomery derives his first main lemma. His result tells us that, as the lattice (shearing) parameter $x$ increases from 0 to $\tfrac{1}{2}$, the value of the 2-dimensional lattice theta function $\theta_\L(\alpha)$ decreases. We will now establish a complementary lemma for our family of theta functions by partially mimicking Montgomery's proof. However, we have to put more effort in achieving our result because the duality given by the Poisson summation formula is more involved than in Montgomery's case, where the duality only resulted in a re-scaling.
\begin{lemma}\label{lem_Main1_c}
	Let $\alpha \geq 1$ be fixed, $x \in \left(0, \tfrac{1}{2} \right)$ and $y \geq \frac{1}{\sqrt{2}}$, then
	\begin{equation}
		\frac{\partial}{\partial x} \, \theta_\L^c (\alpha) > 0.
	\end{equation}
\end{lemma}

\begin{proof}
	Note that we can write $\theta_\L^c$ in terms of $\vartheta$ defined by \eqref{eq_hk_shift} in the following way;
	\begin{equation}
		\theta_\L^c(\alpha) = \sum_{l \in \Z} e^{-\tfrac{\pi \alpha}{y} y^2 \left(l+\tfrac{1}{2} \right)^2}
		\underbrace{\sum_{k \in \Z} e^{-\tfrac{ \pi \alpha}{y} \left( k + \left( \tfrac{1}{2} + x \left(l + \tfrac{1}{2} \right) \right) \right)^2}}_{= \vartheta\left( \tfrac{1}{2} + x \left(l + \tfrac{1}{2} \right); \tfrac{\alpha}{y} \right)}.
	\end{equation}
	Differentiation of $\theta_\L^c$ with respect to $x$ yields
	\begin{align}\label{eq_diff_x}
		\frac{\partial}{\partial x} \, \theta_\L^c(\alpha)
		& = \sum_{l \in \Z} e^{- \pi \alpha y \left( l + \tfrac{1}{2} \right)^2} \left(l + \tfrac{1}{2} \right) \, \vartheta' \left( \tfrac{1}{2} + x \left(l + \tfrac{1}{2} \right); \tfrac{\alpha}{y} \right),
	\end{align}
	where $\vartheta'$ denotes the derivative of $\vartheta$ with respect to the first argument. Note that $\vartheta'(\beta;t)$ is an odd function of $\beta$ with period 1. By pairing the indices $l$ and $-l-1$, we can now re-write the series as
	\begin{equation}
		2 \sum_{l \geq 0} \left( l + \tfrac{1}{2} \right) e^{-\pi \alpha y \left( l + \tfrac{1}{2} \right)^2} \, \vartheta' \left( \tfrac{1}{2} + x \left(l + \tfrac{1}{2} \right); \tfrac{\alpha}{y} \right).
	\end{equation}
	We want to show that the above series (ignoring the factor $2$) is positive for $x \in \left(0, \tfrac{1}{2} \right)$.  By combining \eqref{eq_hk_Fourier} and Lemma \ref{lem_aux2_Mont}, we see that the above expression is certainly no less than
	\begin{align}
		\sqrt{\tfrac{y}{\alpha}} \, & \tfrac{1}{2} \, e^{-\pi \tfrac{\alpha y}{4}} A \left( \tfrac{y}{\alpha} \right) \underbrace{\left( - \sin\left( 2 \pi \left( \tfrac{1}{2} + \tfrac{x}{2} \right) \right) \right)}_{= \sin( \pi x) > 0}\\
		& - \sqrt{\tfrac{y}{\alpha}} \, \sum_{l \geq 1} \left( l + \tfrac{1}{2} \right) e^{-\pi \alpha y \left( l + \tfrac{1}{2} \right)^2} B\left( \tfrac{y}{\alpha} \right) \underbrace{\left| \sin\left( 2 \pi \left( \tfrac{1}{2} + x \left(l + \tfrac{1}{2} \right) \right) \right) \right|}_{=\left| \sin\left( 2 \pi x \left(l + \tfrac{1}{2} \right) \right) \right|}.
	\end{align}
	Showing that this is positive is equivalent to showing that
	\begin{equation}
		\frac{A\left( \tfrac{y}{\alpha} \right)}{B\left( \tfrac{y}{\alpha} \right)} > 2 \, \sum_{l \geq 1} \left( l + \tfrac{1}{2} \right) e^{-\pi \alpha y \left( l^2 + l \right)} \frac{\left| \sin\left( 2 \pi x \left(l + \tfrac{1}{2} \right) \right) \right|}{\sin( \pi x)}, \qquad x \in \left(0, \tfrac{1}{2} \right).
	\end{equation}
	We will use the following fact
	\begin{equation}
		\left| \frac{\sin\left( 2 \pi x \left(l + \tfrac{1}{2} \right) \right)}{\sin( \pi x)}\right| \leq 2 \left| l + \tfrac{1}{2} \right|, \qquad l \in \Z.
	\end{equation}
	For $l \in \N_0$, the above estimate follows by induction (it is certainly true for $l = 0$) and using the addition theorems for sine and cosine. The result for $l \in \Z \backslash \N_0$ follows from the symmetry of the sine function.
	
	Under the assumption that $x \in \left(0, \tfrac{1}{2} \right)$, we succeed in proving the statement of this lemma if we can show that
	\begin{equation}\label{eq_proof_x}
		\frac{A\left( \tfrac{y}{\alpha} \right)}{B\left( \tfrac{y}{\alpha} \right)} > 4 \, \sum_{l \geq 1} \left( l + \tfrac{1}{2} \right)^2 e^{-\pi \alpha y \left( l^2 + l \right)}.
	\end{equation}
	We have to distinguish 2 cases.
	
	\textit{First case} ($\alpha > y$): By the assumption that $y \geq \tfrac{1}{\sqrt{2}}$, we get that
	\begin{equation}
		\alpha y = \frac{\alpha}{y} y^2 \geq \frac{\alpha}{2 y}.
	\end{equation}
	For $\tfrac{y}{\alpha} < 1$, the left-hand side of \eqref{eq_proof_x} becomes
	\begin{equation}
		\frac{A\left( \tfrac{y}{\alpha} \right)}{B\left( \tfrac{y}{\alpha} \right)} = e^{- \tfrac{\pi}{4 \tfrac{y}{\alpha}}} = e^{- \tfrac{\pi \alpha}{4 y}},
	\end{equation}
	whereas we can estimate the right-hand side of \eqref{eq_proof_x} by
	\begin{align}
		4 \, \sum_{l \geq 1} \left( l + \tfrac{1}{2} \right)^2 e^{-\tfrac{\pi \alpha}{2 y} \left( l^2 + l \right)}
		& = 4 \, e^{- \tfrac{\pi \alpha}{4 y}} \, \sum_{l \geq 1} \left( l + \tfrac{1}{2} \right)^2 e^{-\tfrac{\pi \alpha}{2 y} \left( l^2 + l - \tfrac{1}{2}\right)}\\
		&  {<} 4 \, e^{- \tfrac{\pi \alpha}{4 y}} \, \sum_{l \geq 1} \left( l + \tfrac{1}{2} \right)^2 e^{-\tfrac{\pi}{2} \left( l^2 + l - \tfrac{1}{2}\right)}.
	\end{align}
	Computing the value of the series yields
	\begin{equation}
		4 \sum_{l \geq 1} \left( l + \tfrac{1}{2} \right)^2 e^{-\tfrac{\pi}{2} \left( l^2 + l - \tfrac{1}{2}\right)} = 0.857448 \ldots < 1.
	\end{equation}
	This shows that \eqref{eq_proof_x} holds in this case.
	
	\textit{Second case} ($\alpha \leq y$): Now, the left-hand side of \eqref{eq_proof_x} becomes
	\begin{equation}
		\frac{A\left( \tfrac{y}{\alpha} \right)}{B\left( \tfrac{y}{\alpha} \right)} = \frac{1-\tfrac{1}{3000}}{1+\tfrac{1}{3000}} {=\frac{2999}{3001}}.
	\end{equation}
	By assumption,
	\begin{equation}
		\alpha y \geq \alpha^2 \geq 1.
	\end{equation}
	Hence, the right-hand side of \eqref{eq_proof_x} can be estimated by
	\begin{equation}\label{eq_2999_3001}
		4 \, \sum_{l \geq 1} \left( l + \tfrac{1}{2} \right)^2 e^{-\pi \alpha y \left( l^2 + l \right)}  {\leq} 4 \, \sum_{l \geq 1} \left( l + \tfrac{1}{2} \right)^2 e^{-\pi \left( l^2 + l \right)} = 0.0808504 \ldots {< \frac{2999}{3001}} \, ,
	\end{equation}
	which shows that \eqref{eq_proof_x} also holds in this case.
\end{proof}

We note that, in the proof, we use the assumption that $\alpha \geq 1$ only at the very end in equation \eqref{eq_2999_3001}. By trial and error, we find out that the proof also works for $\alpha \geq \tfrac{21}{50} = 0.42$. This is already close to the optimal bound for $\alpha$ with the above methods and estimates.

We also want to show that the result holds for $\alpha < 1$. In \cite{Montgomery_Theta_1988}, the functional equation \eqref{eq_functional} and the fact that $\Theta_\L(0,0;\alpha) = \widehat{\Theta}_\L(0,0;\alpha)$, allowed to conclude that

\begin{equation}
	\Theta_\L(0,0;\alpha) = \tfrac{1}{\alpha} \Theta_\L(0,0,\tfrac{1}{\alpha}),
	\quad \text{ or, equivalently } \quad
	\widehat{\Theta}_\L(0,0;\alpha) = \tfrac{1}{\alpha} \widehat{\Theta}_\L(0,0,\tfrac{1}{\alpha}).
\end{equation}

Hence, in \cite{Montgomery_Theta_1988} the result for $\alpha \geq 1$ already implies the result for $\alpha < 1$ (and vice versa).

However, we need to use the functional equation for $\theta_\L^c (\alpha) = \Theta_\L(\tfrac{1}{2}, \tfrac{1}{2}; \alpha)$ and $\theta_\L^\pm (\alpha) = \widehat{\Theta}_\L(\tfrac{1}{2}, \tfrac{1}{2}; \alpha)$, which do not coincide. Therefore, we need to establish another auxiliary result.

\begin{lemma}\label{lem_Main1_pm}
	Let $\alpha \geq 1$ be fixed, $x \in \left(0, \tfrac{1}{2} \right)$ and $y \geq \frac{1}{\sqrt{2}}$, then
	\begin{equation}
		\frac{\partial}{\partial x} \, \theta_\L^\pm (\alpha) > 0.
	\end{equation}
\end{lemma}
\begin{proof}
	We write $\theta_\L^\pm$ in terms of $\vartheta_2$ defined by \eqref{eq_theta2+-} in the following way;
	\begin{align}
		\theta_\L^\pm (\alpha) & = \sum_{l \in \Z} (-1)^l e^{- \pi \alpha y l^2} \underbrace{\sum_{k \in \Z} (-1)^k e^{- \tfrac{\pi \alpha}{y} (k + x l)^2}}_{= \vartheta_2 \left( x l; \tfrac{\alpha}{y} \right)}\\
		& = \vartheta_2 \left(0; \tfrac{\alpha}{y} \right) + 2 \sum_{l \geq 1} (-1)^l e^{- \pi \alpha y l^2} \vartheta_2 \left( x l; \tfrac{\alpha}{y} \right).
	\end{align}
	Here, we used the fact that $\vartheta_2(\beta;t)$ is an even function of $\beta$. Computing the derivative with respect to $x$ yields
	\begin{equation}\label{eq_theta_pm_derivative}
		\dfrac{\partial}{\partial x} \theta_\L^\pm (\alpha) = 2 \sum_{l \geq 1} (-1)^l e^{- \pi \alpha y l^2} l \, \vartheta_2' \left( x l; \tfrac{\alpha}{y} \right),
	\end{equation}
	where $\vartheta_2'$ denotes the derivative of $\vartheta_2$ with respect to the first argument and we used the fact that $\vartheta_2' \left(0; \tfrac{\alpha}{y} \right) = 0$.
	
	Next, we observe, by using Lemma \ref{lem_aux2} combined with the functional equation \eqref{eq_theta2_functional}, that the series in \eqref{eq_theta_pm_derivative} (ignoring the factor 2) is at least
	\begin{equation}
		\sqrt{\tfrac{y}{\alpha}} \, A_2 \left( \tfrac{y}{\alpha} \right) \sin( \pi x ) \, e^{- \pi \alpha y} - \sqrt{\tfrac{y}{\alpha}} \sum_{l \geq 2} B_2 \left( \tfrac{y}{\alpha} \right) \, |\sin( \pi x l)| \, l \, e^{- \pi \alpha y l^2},
	\end{equation}
	which we wish to show is positive. We note that, for $x \in \left( 0, \tfrac{1}{2}\right)$, we have, by induction and using the addition theorems for sine and cosine, that
	\begin{equation}
		\frac{|\sin(\pi l x)|}{\sin(\pi x)} \leq l, \qquad l \in \N.
	\end{equation}
	Therefore, we prove that the expression in \eqref{eq_theta_pm_derivative} is positive, if we can show that
	\begin{equation}\label{eq_proof_xpm}
		\frac{A_2 \left( \tfrac{y}{\alpha} \right)}{B_2 \left( \tfrac{y}{\alpha} \right)} > \sum_{l \geq 2} l^2 e^{-\pi \alpha y (l^2 - 1)}.
	\end{equation}
	We distinguish 2 cases
	
	\textit{First case} ($\alpha > y)$:
	By assumption $y \geq \tfrac{1}{\sqrt{2}}$, so
	\begin{equation}
		\alpha y = \frac{\alpha}{y} y^2 \geq \frac{\alpha}{2 y}.
	\end{equation}
	As $\tfrac{y}{\alpha} < 1$, the left-hand side of \eqref{eq_proof_xpm} is 
	\begin{equation}
		\frac{A_2 \left( \tfrac{y}{\alpha} \right)}{B_2 \left( \tfrac{y}{\alpha} \right)} 
%		= 2 \left(1 - \tfrac{1}{175} \right) e^{-\tfrac{\pi}{4 \tfrac{y}{\alpha}}} 
		= 2 \left(1 - \tfrac{1}{175} \right) e^{- \tfrac{\pi \alpha}{4 y}}.
	\end{equation}
	We estimate the right-hand side of \eqref{eq_proof_xpm} in the following way;
	\begin{equation}
		\sum_{l \geq 2} l^2 e^{-\pi \alpha y (l^2 - 1)} \leq \sum_{l \geq 2} l^2 e^{- \tfrac{\pi \alpha}{2 y} (l^2 - 1)} = e^{- \tfrac{\pi \alpha}{4 y}} \sum_{l \geq 2} l^2 e^{- \tfrac{\pi}{2} \left(l^2 - \tfrac{3}{2} \right)}.
%		\geq e^{- \tfrac{\pi \alpha}{4 y}} \sum_{l \geq 2} l^2 e^{- \tfrac{\pi \alpha}{2 y} \left(l^2 - \tfrac{3}{2} \right)}.
	\end{equation}
	In this case, the result now follows because
	\begin{equation}
		2 \left(1 - \tfrac{1}{175} \right) > 0.0788803 \ldots = \sum_{l \geq 2} l^2 e^{- \tfrac{\pi}{2} \left(l^2 - \tfrac{3}{2} \right)}.
	\end{equation}
	
	\textit{Second case} ($\alpha \leq y$):
	In this case we use the assumption that $\alpha \geq 1$, which gives the fact that 
	\begin{equation}
		\alpha y \geq \alpha^2 \geq 1.
	\end{equation}
	The left-hand side of \eqref{eq_proof_xpm} is
	\begin{equation}
		\frac{A_2 \left( \tfrac{y}{\alpha} \right)}{B_2 \left( \tfrac{y}{\alpha} \right)} = \frac{1- \tfrac{1}{175}}{1+ \tfrac{1}{55}} > 0.97,
	\end{equation}
	whereas for the right-hand side of \eqref{eq_proof_xpm} we have
	\begin{equation}
		\sum_{l \geq 2} l^2 e^{- \pi \alpha y (l^2-1)} \leq \sum_{l \geq 2} l^2 e^{- \pi (l^2-1)} < 0.0003228.
	\end{equation}
	This finishes the proof.
\end{proof}
\begin{corollary}[First Main Lemma]\label{cor_Main1}
	Let $\alpha > 0$ be fixed, $x \in \left(0, \tfrac{1}{2} \right)$ and $y \geq \tfrac{1}{\sqrt{2}}$, then
	\begin{equation}
		\dfrac{\partial}{\partial x} \theta_\L^c(\alpha) > 0
		\qquad \textnormal{ and } \qquad
		\dfrac{\partial}{\partial x} \theta_\L^\pm (\alpha) > 0.
	\end{equation}
\end{corollary}
\begin{proof}
	The result now follows easily from Lemma \ref{lem_Main1_c} and \ref{lem_Main1_pm} by using the functional equation
	\begin{equation}
		\theta_\L^c (\alpha) = \tfrac{1}{\alpha} \theta_\L^\pm \left( \tfrac{1}{\alpha} \right).
	\end{equation}
\end{proof}

\section{The second main lemma}\label{sec_Main2}

Now, we want to show our second main lemma, stating that the partial derivative with respect to $y$ of the centered lattice theta function is negative, which implies the maximality of the hexagonal lattice by combining it with the previous lemma. More precisely, we show the following result.
\begin{lemma}[Second Main Lemma]\label{lem_Main2}
	Let $\alpha > 0$ be fixed, $x = \frac{1}{2}$ and $y \geq \frac{\sqrt{3}}{2}$, then
	\begin{equation}
		\frac{\partial}{\partial y} \, \theta_\L^\pm(\alpha) < 0
		\qquad \textnormal{ and } \qquad
		\frac{\partial}{\partial y} \, \theta_\L^c(\alpha) < 0.
	\end{equation}
\end{lemma}
\begin{proof}
	It is sufficient to show the lemma for $\theta_\L^c(\alpha)$ since the other result follows by the functional equation 
	\begin{equation}
		\theta_\L^c (\alpha) = \tfrac{1}{\alpha} \theta_\L^\pm \left( \tfrac{1}{\alpha} \right).
	\end{equation}
	For $\alpha > 0$ and $x = \frac{1}{2}$, we can write
	\begin{align}
		\theta^c_\L(\alpha) & = \sum_{k,l \in \Z} e^{- \tfrac{\pi \alpha}{y} \left( \left( k+\tfrac{1}{2} \right)^2 + \left( k+\tfrac{1}{2} \right)\left( l+\tfrac{1}{2} \right) + \left(\tfrac{1}{4} + y^2 \right) \left( l+\tfrac{1}{2} \right)^2 \right)}\\
		& = \sum_{k,l \in \Z} e^{- \pi \alpha \left( \tfrac{1}{y} \left( k + \tfrac{1}{2} + \frac{l + \tfrac{1}{2}}{2} \right)^2 + y \left( l + \tfrac{1}{2} \right)^2 \right)}\\
		& = \sum_{k,l \in \Z} e^{- \pi \alpha \left( \tfrac{1}{4y} \left( 2k + 1 + l + \tfrac{1}{2} \right)^2 + y \left( l + \tfrac{1}{2} \right)^2 \right)}.
	\end{align}
	By shifting the index $l$ by 1, we get
	\begin{align}
		\theta^c_\L(\alpha) & = \sum_{k,l \in \Z} e^{- \pi \alpha \left( \tfrac{1}{4y} \left( 2k + l + \tfrac{1}{2} \right)^2 + y \left( l - \tfrac{1}{2} \right)^2 \right)}\\
		& = \sum_{l \in \Z} \left( e^{- \pi \alpha y \left( l - \tfrac{1}{2} \right)^2} \sum_{k \in \Z} e^{-\pi \alpha \tfrac{1}{4 y} \left( 2k + l + \tfrac{1}{2}\right)^2} \right).
	\end{align}
	We will now distinguish the cases where $l$ is even and where $l$ is odd. Then, we can write
	\begin{align}
		\theta^c_\L(\alpha) & = \sum_{l' \in \Z} e^{-\pi \alpha y \left( 2l' - \tfrac{1}{2} \right)^2} \sum_{k \in \Z} e^{-\pi \alpha \tfrac{1}{4 y} \left( 2(k+l') + \tfrac{1}{2} \right)^2}\\
		& \quad + \sum_{l' \in \Z} e^{-\pi \alpha y \left( 2l' + 1 - \tfrac{1}{2} \right)^2} \sum_{k \in \Z} e^{-\pi \alpha \tfrac{1}{4 y} \left( 2(k+l') + 1 + \tfrac{1}{2} \right)^2}\\
		& = \sum_{l' \in \Z} e^{-\pi 4 \alpha y \left( l' - \tfrac{1}{4} \right)^2} \sum_{k \in \Z} e^{-\pi \alpha \tfrac{1}{y} \left( k+l' + \tfrac{1}{4} \right)^2}\\
		& \quad + \sum_{l' \in \Z} e^{-\pi 4 \alpha y \left( l' + \tfrac{1}{4} \right)^2} \sum_{k \in \Z} e^{-\pi \alpha \tfrac{1}{y} \left( k+l' + \tfrac{3}{4} \right)^2}.
	\end{align}
	By introducing the new summation index $k'=k+l'$, which runs through all of $\Z$, we obtain
	\begin{align}
		\theta^c_\L(\alpha) & = \sum_{l' \in \Z} e^{-\pi 4 \alpha \left( l' - \tfrac{1}{4} \right)^2} \sum_{k' \in \Z} e^{-\pi \alpha \tfrac{1}{y} \left( k' + \tfrac{1}{4} \right)^2}\\
		& \quad + \sum_{l' \in \Z} e^{-\pi 4 \alpha \left( l' + \tfrac{1}{4} \right)^2} \sum_{k' \in \Z} e^{-\pi \alpha \tfrac{1}{y} \left( k' + \tfrac{3}{4} \right)^2}\\
		& = \vartheta \left( -\tfrac{1}{4}; 4 \alpha y \right) \vartheta \left( \tfrac{1}{4}; \tfrac{\alpha}{y} \right) + \vartheta \left( \tfrac{1}{4}; 4 \alpha y \right) \vartheta \left( \tfrac{3}{4}; \tfrac{\alpha}{y} \right),
	\end{align}
	where we used the notation from \eqref{eq_hk_shift}. By using Lemma \ref{lem_1/4} and Lemma \ref{lem_1/4_theta2} we obtain
	\begin{equation}\label{eq_thetacy1/2}
		\theta^c_\L(\alpha) = 2 \, \vartheta \left( \tfrac{1}{4}; 4 \alpha y \right) \vartheta \left( \tfrac{1}{4}; \tfrac{\alpha}{y} \right) = \frac{1}{2} \, \theta_2(\alpha y) \theta_2 \left( \tfrac{\alpha}{4y} \right).
	\end{equation}
	By a change of variables (for instance $z=2y$) and the result in \cite{FaulhuberSteinerberger_Theta_2017}, which is Proposition \ref{pro_fauste} in this work, we see that the last expression is maximal for $y = \tfrac{1}{2}$ (i.e., $z=1$) and that it is decreasing for $y > \tfrac{1}{2}$ (and increasing before). As $y \geq \tfrac{\sqrt{3}}{2} > \frac{1}{2}$ by assumption, the proof is finished. 
\end{proof}
We believe that Lemma \ref{lem_Main2} remains true if we replace the condition $x = \frac{1}{2}$ by $x \in [0, \tfrac{1}{2}]$ ($x$ fixed) and $y^2 \geq 1 - x^2$. For $x = 0$ the statement holds by the results of Faulhuber and Steinerberger \cite{FaulhuberSteinerberger_Theta_2017} and our numerics support the above statement for some other values of $x$. This would give a complete picture of the behavior of the studied family of lattice theta functions. However, we do not need a more general result in order to prove Theorem \ref{thm_main}.

\section{Proof of Main Result}
We will briefly collect the results we obtained so far and stack them together to a complete proof of Theorem \ref{thm_main}. We start by noting that
\begin{equation}
	\theta_{\mathsf{A}_2}^c(\alpha) \geq \theta_\L^c(\alpha), \quad \alpha > 0.
\end{equation}
In order to see this we note that the First Main Lemma (Corollary \ref{cor_Main1}) tells us that the parameters of the maximizing lattice
\begin{equation}
	\L = y^{-1/2}
	\begin{pmatrix}
		1 & x\\
		0 & y
	\end{pmatrix} \Z^2
\end{equation}
must be found on the right (and due to symmetry on the left) boundary line of the fundamental domain $D$ defined in \eqref{eq_D}, i.e., $x = \frac{1}{2}$. Since we are on the right boundary line of $D$, we necessarily have $y \geq \frac{\sqrt{3}}{2}$. Now, by the Second Main Lemma (Lemma \ref{lem_Main2}), the value of $\theta_\L^c(\alpha)$ is decreasing for $x = \frac{1}{2}$ and $y \geq \frac{\sqrt{3}}{2}$ increasing. Hence, the lattice maximizing $\theta_\L^c(\alpha)$ has parameters $(x,y) = (\frac{1}{2}, \frac{\sqrt{3}}{2})$, which gives the hexagonal lattice. By the preliminaries in Section \ref{sec_Intro} (see also Section \ref{sec:parameters}), the only other lattices maximizing $\theta_\L^c$ are obtained by choosing a different basis or by rotation.

For the function $\theta_\L^\pm(\alpha)$, we now actually only have to use the functional equation
\begin{equation}
	\theta_\L^\pm(\alpha) = \tfrac{1}{\alpha} \theta_\L^c(\tfrac{1}{\alpha}).
\end{equation}
Nonetheless, we had to use the functional equation already earlier in order to deduce the result for $\theta_\L^c$, which justifies stating the result for $\theta_\L^\pm$ additionally.

\section{Universal Optimality of the Hexagonal Lattice}\label{sec_univopt}
We are now going to prove Corollary \ref{cor_cm}, which yields the universal optimality of the hexagonal lattice for Madelung-like lattice energies. We basically follow the lines of \cite[Prop. 3.1]{Bet16}. The result follows from Theorem \ref{thm_main} and an application of the Hausdorff-Bernstein-Widder Theorem \cite{bernstein}. The latter states that $f$ is completely monotone if and only if $f$ is the Laplace transform of a non-negative Borel measure $\mu_f$, i.e. 
\begin{equation}\label{eq_Laplace}
	f(r)=\int_0^\infty e^{-rt} d\mu_f(t).
\end{equation}
The condition $|f(r)| = O(r^{-1-\varepsilon})$ ensures the absolute summability of $E_f^\pm[\L]$ and $E_f^c[\L]$ for any 2-dimensional lattice $\L$. This implies that 
\begin{equation}\label{eq_representtheta}
	E_f^\pm[\L]=\int_0^\infty \left( \theta_\L^\pm \left( \tfrac{\alpha}{\pi} \right) - 1 \right) \, d\mu_f(\alpha),\quad E_f^c[\L]=\int_0^\infty \theta_\L^c \left( \tfrac{\alpha}{\pi} \right)  \, d\mu_f(\alpha).
\end{equation}
Since $\mu_f\geq 0$, the optimality proved in Theorem \ref{thm_main} for the  2-dimensional alternating and centered lattice theta functions $\theta_\L^\pm(\alpha)$ and $\theta_\L^c(\alpha)$, for any value of the parameter $\alpha>0$, implies the same optimality of the above lattice energies.

For showing that $E_f^\pm[\L]<0$ for all completely monotone function $f$ and all 2-dimensional lattices $\L$ of unit volume, it is sufficient, according to \eqref{eq_representtheta} and Theorem \ref{thm_main}, to prove that
$$
\theta_{\mathsf{A}_2}^\pm \left( \alpha \right) < 1, \quad \forall \alpha>0.
$$
This statement is equivalent to
\begin{equation}
	\theta_{\mathsf{A}_2}^c(\alpha)<\frac{1}{\alpha}, \quad \forall \alpha>0,
\end{equation}
via the functional equation \eqref{eq_fcteqtheta}. From \eqref{eq_thetacy1/2} applied to $y=\frac{\sqrt{3}}{2}$, we know that
$$
\theta_{\mathsf{A}_2}^c(\alpha)=\frac{1}{2}\theta_2\left(\frac{\sqrt{3}\alpha}{2}  \right)\theta_2\left(\frac{\alpha}{2\sqrt{3}}  \right).
$$
We now write $\theta_2$ in terms of $\theta_4$ using the functional equation \eqref{eq_fcteqtheta24} and we obtain
$$
\theta_{\mathsf{A}_2}^c(\alpha)=\frac{1}{\alpha}\theta_4\left( \frac{2}{\sqrt{3}\alpha}\right)\theta_4\left(\frac{2\sqrt{3}}{\alpha} \right)<\frac{1}{\alpha}.
$$
By using the product representation (see, e.g., \cite{WhiWat69})
\begin{equation}
	\theta_4(t) = \prod_{k \geq 1} (1-e^{-2 \pi k t})(1-e^{-(2k-1) \pi t})^2
\end{equation}
we see that $\theta_4(t)<1$ for all $t>0$, which completes the proof.

We remark that Corollary \ref{cor_cm} holds for any completely monotone function $f$ without integrability restriction, if we re-define the energies $E_f^\pm$ and $E_f^c$ using the Ewald summation method. This has, for instance, been carried out in \cite{BetKnu_Born_18}, by writing
	\begin{equation}
		\widetilde{E}_f^\pm[\L]:=\lim_{\varepsilon \to 0}\sum_{\substack{(k,l)\in \Z^2\\(k,l)\neq (0,0)}} (-1)^{k+l}f(|kv_1+l v_2|^2)e^{-\varepsilon |kv_1+l v_2|^2 } 
	\end{equation}
	and 
	\begin{equation}
		\widetilde{E}_f^c[\L]:=\lim_{\varepsilon \to 0}\sum_{(k,l)\in \Z^2 } f(|kv_1+l v_2+c|^2)e^{-\varepsilon |kv_1+l v_2+c|^2 }.
	\end{equation}
	Therefore, by \eqref{eq_Laplace} it is possible to write both energies in terms of the alternating and centered lattice theta functions.
	%in a bit more complicated way than in \eqref{eq_representtheta} and
	Then, by using the optimality of the hexagonal lattice for all $\alpha>0$, we arrive at the desired optimality result. However, even though the summation for computing $\widetilde{E}_f^\pm$ and $\widetilde{E}_f^c$ coincides with $E_f^\pm$ and $E_f^c$ when $f$ satisfies the integrability assumption of Corollary \ref{cor_cm}, the conditional summability in the non-integrable case means that, maybe, different optimality results could hold if we use another summation method. That is why we have chosen to state our corollary only in the absolutely summable case.

Furthermore, Corollary \ref{cor_cm} directly applies to the potential $f(r)=r^{-s/2}$ where $s>2$. We call the resulting energies the alternating and centered Epstein zeta functions, defined by
\begin{equation}\label{eq_Epstein}
	\zeta_\L^\pm(s):=\sum_{\substack{(k,l)\in \Z^2\\(k,l)\neq (0,0)}} \frac{(-1)^{k+l}}{|k v_1+l v_2|^{s}}
	\quad \textnormal{and} \quad
	\zeta_\L^c(s):=\sum_{(k,l)\in \Z^2}\frac{1}{|k v_1+l v_2+c|^{s}}.
\end{equation}
The maximality of these modifications of the Epstein zeta function for the hexagonal lattice is included in our Corollary \ref{cor_cm}. The alternating Epstein zeta function $\zeta_\L^\pm(s)$, defined by \eqref{eq_Epstein}, already appeared in \cite{BFK20} where we investigated the optimality of the \textit{rock-salt structure} among lattices with an alternating distribution of charges and inverse power laws interactions. The lattice energy under consideration in \cite{BFK20} has the form
\begin{equation}
	E_{p,q,\rho}[\L]:= \zeta_\L(p) + \rho^{-1}\zeta_\L^\pm(q),\quad p>q.
\end{equation}
Here, the factor $\rho$ reflects the density of the system. Since $\zeta_\L^\pm(s)$ is the main term of $E_{p,q,\rho}$ as $\rho \to 0$, Corollary \ref{cor_cm} rigorously shows that the hexagonal lattice is the maximizer of this type of energy model in the low density limit. We observed this fact numerically already in \cite[Fig.~7(d)]{BFK20}. Hence, Corollary \ref{cor_cm} contains important information about ionic crystal energies when the charges are alternating.

\section{Appendix: More about parametrization of Lattices}\label{sec:parameters}
In this section we are going to explain how 2-dimensional lattices can be parametrized by one complex number in the Siegel upper half-plane $\H$, and, also, what are the precise geometrical meanings of the parameters $x$ and $y$. We mostly follow the textbook by Serre \cite[Chap.~VII §1]{Serre} in the beginning.

The Siegel upper half-plane is defined as
\begin{equation}
	\H = \{ \tau \in \C \mid \Im(\tau) > 0\}.
\end{equation}
We will write elements in $\H$ as $\tau = x + i y$, where $x$ and $y$ are the real an imaginary part respectively. Also, we will identify $\tau \in \H$ with the vector $(x,y) \in \R^2$ in the canonical way. Recall that, even though we write $(x,y) \in \R^2$, we actually deal with column vectors.

A (full-rank) lattice in $\C$ of volume 1, which can naturally be identified with a lattice in $\R^2$, is given by
\begin{equation}
	\mathcal{L} = |\Im(\overline{\omega_1} \omega_2)|^{-1/2} \left(\omega_1 \Z \times \omega_2 \Z\right),
\end{equation}
with the condition that $\frac{\omega_1}{\omega_2} \notin \R$. The factor $|\Im(\omega_1 \overline{\omega_2})|^{-1/2}$ normalizes the lattice to have density 1. By identifying lattices which result from one another by rotation, we can always find a representative of the form
\begin{equation}
	\mathcal{L}_\tau = \Im(\tau)^{-1/2} \left(\Z \times \tau \Z\right) \subset \C .
\end{equation}

In $\R^2$, a lattice of unit volume is characterized by a matrix in $SL(2,\R)$. By a QR-decomposition, we may write the lattice as
\begin{equation}
	\L = Q S \Z^2,
\end{equation}
where $Q$ is orthogonal and $S$ is an upper triangular matrix of the form
\begin{equation}
	S = y^{-1/2}
	\begin{pmatrix}
		1 & x\\
		0 & y
	\end{pmatrix}
	\quad \textnormal{ and } \quad
	\L = S \Z^2.
\end{equation}
The parameters $x$ and $y$ are called shearing and dilation parameter, respectively. Again, we identify lattice which only differ from one another by a rotation $Q$ and may write
\begin{equation}
	\L_{(x,y)} = y^{-1/2} \,
	\begin{pmatrix}
		1 & x\\
		0 & y
	\end{pmatrix}
	\Z^2 \subset \R^2.
\end{equation}
Clearly, we have an identification
$
	\mathcal{L}_\tau
	\, \longleftrightarrow \,
	\L_{(x,y)}
$.
\begin{figure}[htb]
	\subfigure[A 2-dimensional lattice and its characteristic parameters. The $\ast$ marks the center of the fundamental cell (gray parallelogram), denoted by $c$ in this work.]{
	\includegraphics[width=.35\textwidth]{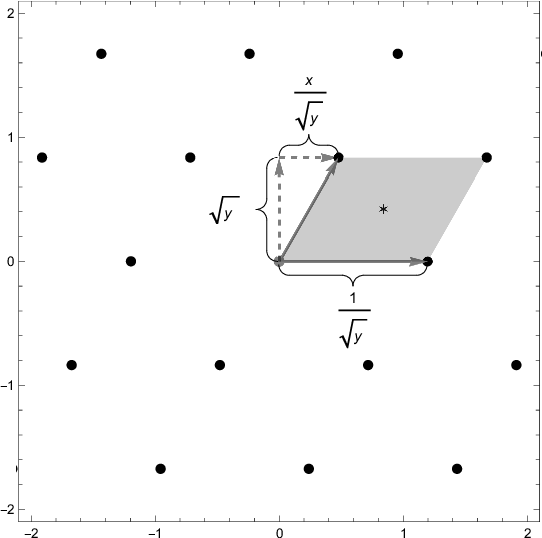}}
	\hfill
	\subfigure[The upper half-plane with the fundamental domain $D$ and some of its copies. The entries show how to obtain the respective domain from $D$, e.g., $T$ means that $T$ has to act on $D$ in order to obtain the domain.]{
	\includegraphics[width=.575\textwidth]{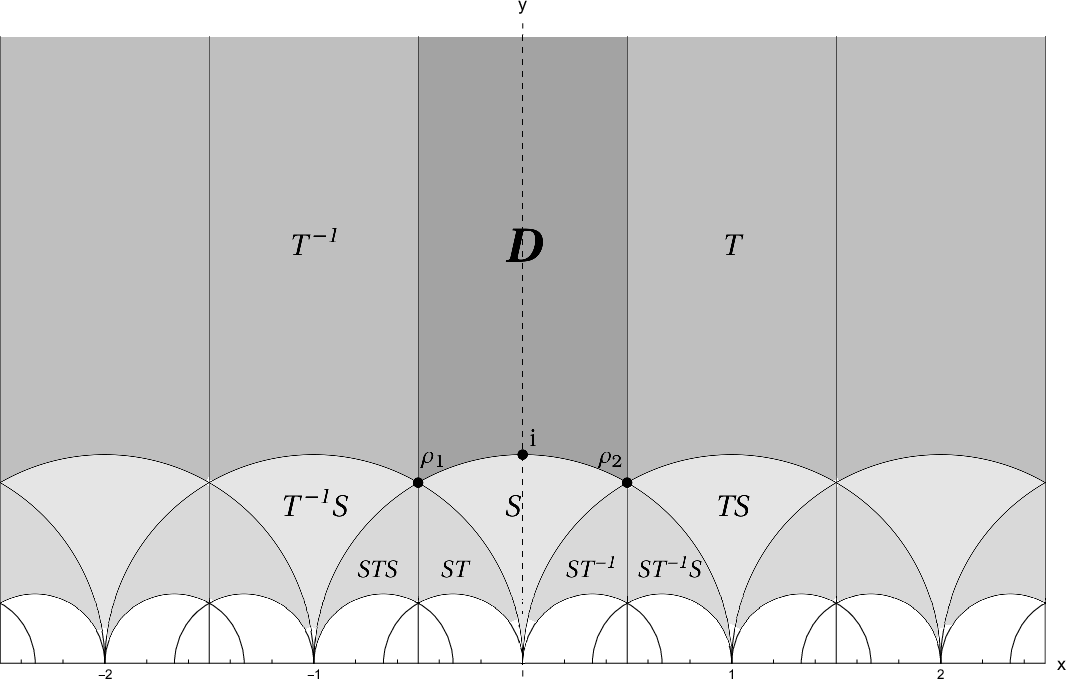}}
	\caption{After a suitable rotation and choice of basis, any lattice $\L_{(x,y)} \subset \R^2$ can be parametrized by a point $\tau \in D \subset \mathbb{H}$.}
	\label{fig_lattice}
\end{figure}

Often, it is necessary to work with a Minkowski or reduced basis. This means that the lattice-generating vectors are the shortest possible. In dimension 2, this is ensured if the generating vectors are derived from a parameter in the so-called fundamental domain $D$, a subset of the upper half-plane $\mathbb{H}$. The fundamental domain $D$ is the following set;
\begin{equation}
	D = \{ \tau \in \mathbb{H} : |\tau| \geq 1, \, |\Re(\tau)| \leq \tfrac{1}{2} \}.
\end{equation}
Due to symmetry reasons, it is sufficient for us to consider the right half of the fundamental domain, i.e., $\{0 \leq \Re(\tau)\} \cap D$.

Next, we make a remark on the modular group $PSL(2, \Z)$, which is the projective special linear group of determinant 1 matrices with integer entries. We note that
\begin{equation}
	PSL(2,\Z) = SL(2,\Z)/\{\pm I\}.
\end{equation}
Geometrically, choosing a matrix $\mathcal{B} \in PSL(2,\Z)$, $\mathcal{B} \neq I$, corresponds to choosing a non-reduced basis for the lattice $\Z^2$. However, we have that $\mathcal{B} \Z^2 = \Z^2$. Also, the group $PSL(2,\Z)$ is generated by the following 2 matrices;
\begin{equation}
	S = \begin{pmatrix}
		0 & 1\\
		-1 & 0
	\end{pmatrix},
	\qquad
	T = \begin{pmatrix}
		1 & 1\\
		0 & 1
	\end{pmatrix}
\end{equation}
Any point $\widetilde{\tau} \in \mathbb{H}$ can be obtained from a unique point $\tau \in D$ (up to the points on the boundary line $\Re(\tau) = \pm \tfrac{1}{2}$) and the action of the matrices $S$ and $T$.

The fundamental domain $D$ and the action of the modular group are illustrated in Figure \ref{fig_lattice}. The special point $i$ corresponds to the square lattice ($x=0$, $y=1$) and the points $\rho_1$ and $\rho_2$ correspond to the triangular lattice ($x=\pm \tfrac{1}{2}$, $y=\tfrac{\sqrt{3}}{2}$). The purely imaginary points, i.e., points on the line $x=0$, give rectangular lattices, meaning that their fundamental cell is a rectangle. Points lying on the boundary of $D$, i.e., $|x| = \tfrac{1}{2}$ or $x^2+y^2=1$, give rhombic lattices, meaning that we can find generating vectors of equal length (which do not necessarily yield a reduced basis). For more details on the modular group and the fundamental domain we refer to the textbook of Serre \cite[Chap.~VII]{Serre}.

\begin{figure}[!htb]
	\subfigure[The square lattice with alternate charge distribution. The fundamental cell is spanned by the standard basis and its center is marked by $\ast$.]{
		\includegraphics[width=.35\textwidth]{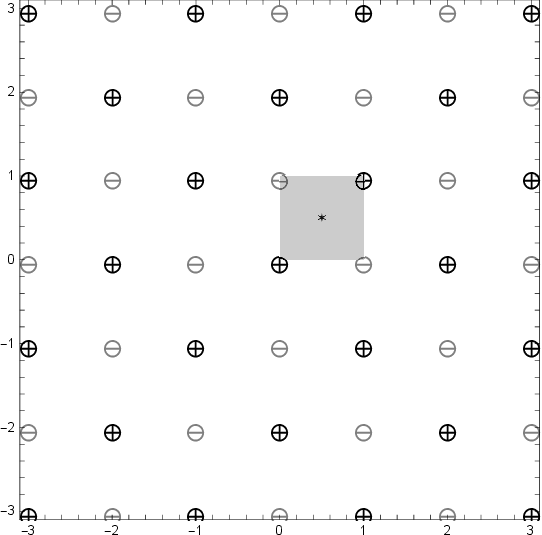}
	}
	\hspace{2cm}
	\subfigure[The square lattice where we chose the basis given by $T$. The charges are alternating with respect to the chosen basis. The fundamental cell is deformed by the action of $T$ and its center is marked by $\ast$.]{
		\includegraphics[width=.35\textwidth]{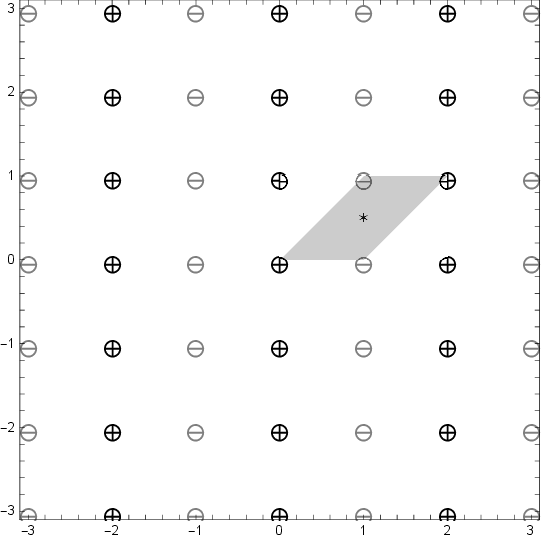}
	}
	\caption{Two different configurations of charges for the square lattice. We used the standard basis and the basis given by $T$. This illustrates the importance of the choice of basis.}
	\label{fig_non_optimal_charge}
\end{figure}
We remark that, by using the full upper half-plane $\mathbb{H}$ to index lattices, our model allows non-optimally charged lattice configurations if we put alternating charges on a lattice with respect to the generating vectors (see Figure \ref{fig_non_optimal_charge}). For example, if we choose $\tau = 1 + i$, which corresponds to the lattice $T \Z^2 = \Z^2$, then the generating vectors are $(1,0)$ and $(1,1)$. Note that they do not form a Minkowski reduced basis. Now, denote the shearing matrix by
\begin{equation}
	V_x = \begin{pmatrix}
		1 & x\\
		0 & 1
	\end{pmatrix},
	\quad x \in \R.
\end{equation}
If $V_x$ acts on the square lattice, i.e., $\L_x = V_x \Z^2$, and we continuously increase the shearing from 0 to 1, we end up with lattice $T \Z^2 = \Z^2$. However, the charged lattice will not be the alternating charged square lattice. Shearing the lattice by $x = 1$, will align the charges on vertical lines and the lines will be alternately charged  as in Figure \ref{fig_non_optimal_charge}. This is due to the fact that for $ x = \tfrac{1}{2}$ we have two competing alternate charge configurations, which both yield the same energy.

The alignment of the charges illustrated in Figure \ref{fig_non_optimal_charge} results from a non-canonical choice of basis for $\Z^2$. If we say that the fundamental cell is the parallelogram spanned by the vectors of the basis, then our model is not independent of the choice of basis because the new center is in general be a shift of the old center by $(m,n) \in \tfrac{1}{2} \Z \times \tfrac{1}{2} \Z$. Nonetheless, the center of a cell will always be the intersection point of the diagonals of the resulting parallelogram. This leads to the idea of using a Delaunay triangulation, introduced by Delaunay in \cite{Del34}, to define non-lattice points with special geometric properties. A Delaunay triangulation $\Delta$ is a special type of triangulation. It is nowadays a common procedure to build a triangular mesh out of a point set, seeking to maximize the smallest angle in all triangles. Delaunay triangulation is widely used in computer graphics and we refer to any modern textbook which treats this topic for further information.

\begin{figure}[!htb]
	\subfigure[Delaunay triangulation for a hexagonal lattice.]{
		\includegraphics[width=.3\textwidth]{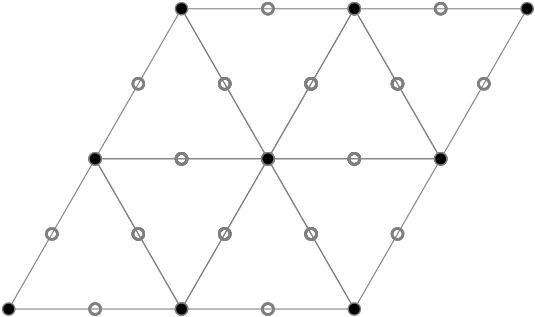}
	}
	\hfill
	\subfigure[Delaunay triangulation for the square lattice.]{
		\includegraphics[width=.19\textwidth]{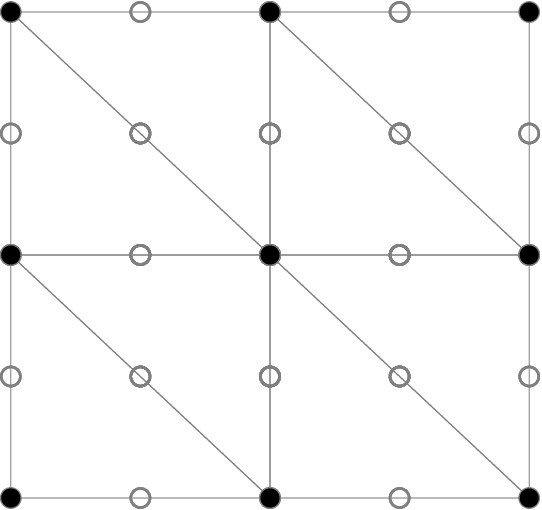}
	}
	\hfill
	\subfigure[Delaunay triangulation for a general lattice.]{
		\includegraphics[width=.3\textwidth]{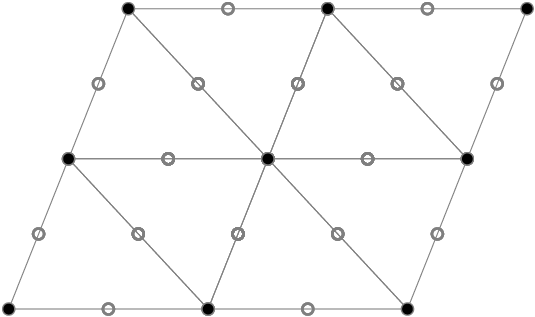}
	}
	\caption{Delaunay triangulations for different lattices. The lattice points are marked by $\bullet$. The \textcolor{gray}{$\circ$} marks the midpoint of an edge between two points which are vertices of a Delaunay triangle. Furthermore, each \textcolor{gray}{$\circ$} is the center of a fundamental cell, depending on the choice of basis we make.}\label{fig_delaunay_latt}
\end{figure}
In Figure \ref{fig_delaunay_latt}, we see Delaunay triangulations for different lattices. We note that, for general point configurations, a Delaunay triangulation is not unique. This can already be seen in Figure \ref{fig_delaunay_latt} (b), where the choice of the diagonal in the square can be made arbitrarily. Together with Figure \ref{fig_delaunay_latt}, this leads us to the conjectures for general point sets of fixed density and alternating and centered $f$-energies.

For a finite general point configuration with an even number of points $N$, we place $\oplus$ and $\ominus$ charges at the points, such that the complete configuration has neutral charge. Then we minimize among all possible (neutral) arrangements for a fixed configuration $X_N$ and seek to find the configuration of points $X_N$ which has least energy at fixed density. For infinite point configurations, we use the thermodynamic limit setting as in \cite[Sect.~9]{CohKum07}. The precise mathematical statement is as follows. For any configuration $X_N=\{x_1,...x_N\}\subset \R^2$, $N\in 2\N$, we define the energy per point by
\begin{equation}
	E_f^\pm[X_N]=\min_{\varphi}\frac{1}{N} \sum_{i\neq j} \varphi(x_i) \varphi(x_j)f(|x_i-x_j|^2),
\end{equation}
where the minimum is taken among charge distributions $\varphi:X_N\to \{-1,1\}$ such that $\sum_{i=1}^N \varphi(x_i)=0$ (neutrality assumption). We therefore conjecture that, for any infinite configuration of points $\mathcal{C}\subset \R^2$ with density 
\begin{equation}\label{eq_densityrho}
\rho:=\lim_{R\to \infty}\frac{\sharp\{ \mathcal{C}\cap B_R\}}{|B_R|}=1,
\end{equation}
where $B_R$ is the ball of radius $R$ centered at 0, and any completely monotone function $f$ such that $|f(r)|=O(r^{-1-\epsilon})$ as $r\to \infty$ for some $\epsilon>0$, we have
\begin{equation}
	\lim_{R\to \infty} E_f^\pm[\mathcal{C}\cap B_R]\leq E_f^\pm[\mathsf{A}_2].
\end{equation}

For the dual problem, i.e., for the energy at the ``center of a cell", we define the following sets. For a finite set $X_N$, we denote the set of all of its Delaunay triangulations $\Delta$ by $\mathcal{D}_{X_N}$. Second, we denote the set of midpoints of a Delaunay triangulation (see Figure \ref{fig_delaunay_gen}) by
\begin{equation}
	\mathcal{M}_\Delta = \{ p = \tfrac{p_1+p_2}{2} \mid p_1, p_2 \text{ are joined by an edge of the Delaunay triangulation } \Delta \in \mathcal{D}_{X_N} \}.
\end{equation}
Now, for any configuration of points $X_N:=\{x_1,...,x_N\}\subset \R^2$, we define
\begin{equation}
	E_f^c[X_N] = \min_{\Delta \in \mathcal{D}} \min_{m \in \mathcal{M}_\Delta} \sum_{x \in X_N} f(|x-m|^2).
\end{equation}
We therefore conjecture that, for any infinite configuration of points $\mathcal{C}\subset \R^2$ with unit density as in \eqref{eq_densityrho}, and any completely monotone function $f$ such that $|f(r)|=O(r^{-1-\epsilon})$ as $r\to \infty$ for some $\epsilon>0$, we have
\begin{equation}
	\lim_{R\to \infty} E_f^c[\mathcal{C}\cap B_R]\leq E_f^c[\mathsf{A}_2].
\end{equation}

\begin{figure}[hbt]
	\includegraphics[width=.5\textwidth]{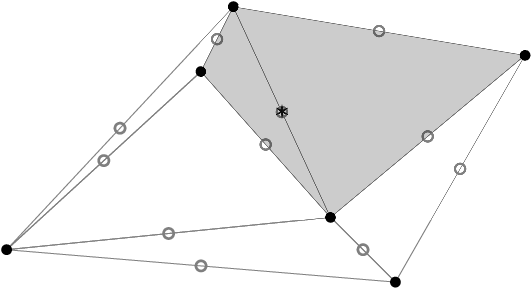}
	\caption{Delaunay triangulation of a (finite) point set. The $\bullet$ mark the points of the configuration and the \textcolor{gray}{$\circ$} mark the midpoints of the edges of a triangle. There is no fundamental cell anymore, but still we may join two triangles with a common edge to form a cell. This cell is not necessarily convex, but it contains a point \textcolor{gray}{$\circ$} (the one on the common edge), which would then be the ``center $\ast$ of the cell".}\label{fig_delaunay_gen}
\end{figure}

\bibliographystyle{plain}

\end{document}